\documentclass[reqno,10pt,centertags]{amsart}
\usepackage{amsmath,amsthm,amscd,amssymb,latexsym,upref,mathrsfs}
\usepackage{esint}

\usepackage{hyperref}
\newcommand*{\mailto}[1]{\href{mailto:#1}{\nolinkurl{#1}}}
\newcommand{\arxiv}[1]{\href{http://arxiv.org/abs/#1}{arXiv:#1}}

\newcommand{\bbR}{{\mathbb{R}}}

\newcommand{\bbC}{{\mathbb{C}}}

\newcommand{\cB}{{\mathcal B}}
\newcommand{\cC}{{\mathcal C}}

\newcommand{\cF}{{\mathcal F}}

\newcommand{\cH}{{\mathcal H}}

\newcommand{\no}{\notag}
\newcommand{\lb}{\label}

\newcommand{\ol}{\overline}

\newcommand{\wti}{\widetilde}

\newcommand{\f}{\frac}
\newcommand{\bi}{\bibitem}
\newcommand{\hatt}{\widehat}

\renewcommand{\Re}{\mathop\mathrm{Re}}
\renewcommand{\Im}{\mathop\mathrm{Im}}
\renewcommand{\ge}{\geqslant}
\renewcommand{\le}{\leqslant}

\DeclareMathOperator{\dom}{dom}

\DeclareMathOperator{\tr}{tr}

\DeclareMathOperator*{\slim}{s-lim}

\newcommand{\Sect}{\text{\rm Sect}}
\newcommand{\BIP}{\text{\rm BIP}}

\allowdisplaybreaks \numberwithin{equation}{section}

\newtheorem{theorem}{Theorem}[section]

\newtheorem{corollary}[theorem]{Corollary}
\newtheorem{definition}[theorem]{Definition}
\newtheorem{hypothesis}[theorem]{Hypothesis}

\theoremstyle{remark}
\newtheorem{remark}[theorem]{Remark}

\begin{document}

\title[Operator Bounds Employing Interpolation]{Some
Operator Bounds Employing Complex Interpolation Revisited}

\author[F.\ Gesztesy]{Fritz Gesztesy} 
\address{Department of Mathematics,
University of Missouri, Columbia, MO 65211, USA}
\email{\mailto{gesztesyf@missouri.edu}}
\urladdr{\url{http://www.math.missouri.edu/personnel/faculty/gesztesyf.html}}

\author[Y.\ Latushkin]{Yuri\ Latushkin} 
\address{Department of Mathematics,
University of Missouri, Columbia, MO 65211, USA}
\email{\mailto{latushkiny@missouri.edu}}
\urladdr{\url{http://www.math.missouri.edu/personnel/faculty/latushkiny.html}}

\author[F.\ Sukochev]{Fedor Sukochev} 
\address{School of Mathematics and Statistics, UNSW, Kensington, NSW 2052,
Australia}
\email{\mailto{f.sukochev@unsw.edu.au}}

\author[Y.\ Tomilov]{Yuri Tomilov} 
\address{Faculty of Mathematics and Computer Science, Nicholas
Copernicus University, ul.\  Chopina 12/18, 87-100 Torun, Poland} 
\email{\mailto{tomilov@mat.uni.torun.pl}}

\thanks{Partially supported by the NSF grant DMS-1067929, by the Research Board and the Research Council of the University of Missouri, by the ARC, the NCN grant DEC-2011/03/B/ST1/00407, and by
the EU Marie Curie IRSES program, project ``AOS'', No.\ 318910.}

\dedicatory{Dedicated with great pleasure to Charles Batty on the occasion of
his 60th birthday.}

\date{\today}


\begin{abstract}
We revisit and extend known bounds on operator-valued functions of the type
$$
T_1^{-z} S T_2^{-1+z}, \quad z \in \ol \Sigma = \{z\in\bbC\,|\, \Re(z) \in [0,1]\},
$$
under various hypotheses on the linear operators $S$ and $T_j$,
$j=1,2$. We particularly single out the case of self-adjoint and
sectorial operators $T_j$ in some separable complex Hilbert space
$\cH_j$, $j=1,2$, and suppose that $S$ (resp., $S^*$) is a densely
defined closed operator mapping $\dom(S) \subseteq \cH_1$ into
$\cH_2$ (resp., $\dom(S^*) \subseteq \cH_2$ into $\cH_1$),
relatively bounded with respect to $T_1$ (resp., $T_2^*$). Using
complex interpolation methods, a generalized polar decomposition
for $S$, and Heinz's inequality, the  bounds we establish lead to
inequalities of the following type: Given $k \in (0, \infty)$,
\begin{align*}
& \big\|\ol{T_2^{-z}ST_1^{-1+z}}\big\|_{\cB(\cH_1,\cH_2)}    \no \\
& \quad \leq N_1 N_2 
e^{k (\Im(z))^2 + k \Re(z) [1-\Re(z)] + (4k)^{-1} (\theta_1 + \theta_2)^2}   \no \\
& \qquad \times \big\|ST_1^{-1}\big\|_{\cB(\cH_1,\cH_2)}^{1-\Re(z)} \,
\big\|S^*(T_2^*)^{-1}\big\|_{\cB(\cH_2,\cH_1)}^{\Re(z)},
\quad z \in \ol \Sigma,   
\end{align*}
which also implies, 
\begin{align*}
& \big\|\ol{T_2^{-x}ST_1^{-1+x}}\big\|_{\cB(\cH_1,\cH_2)}
\leq N_1 N_2 e^{(\theta_1 + \theta_2) [x(1-x)]^{1/2}}  \\
& \quad \times \big\|ST_1^{-1}\big\|_{\cB(\cH_1,\cH_2)}^{1-x}
\, \big\|S^*(T_2^*)^{-1}\big\|_{\cB(\cH_2,\cH_1)}^{x},  \quad x \in [0,1],
\end{align*}
assuming that $T_j$ have bounded imaginary powers, that
is, for some $N_j\ge 1$ and $\theta_j \ge 0,$
$$
\big\|T_j^{is}\big\|_{\cB(\cH)} \leq N_j e^{\theta_j |s|}, \quad s \in \bbR, \; j=1,2.
$$
We also derive analogous bounds with $\cB(\cH_1,\cH_2)$ replaced by trace ideals, 
$\cB_p(\cH_1, \cH_2)$, $p \in [1,\infty)$. 
The methods employed are elementary, predominantly relying on Hadamard's 
three-lines theorem and Heinz's inequality. 
\end{abstract}

\maketitle


\section{Introduction}  \lb{s1}

This paper was inspired by an interesting result proved by Lesch in  
Appendix~A to his 2005 paper \cite{Le05}, dealing with uniqueness of spectral flow on 
spaces of unbounded Fredholm operators. More precisely, upon a close inspection 
of the proof of \cite[Proposition\ A.1]{Le05}, we derived the following interpolation 
result in \cite{GLMST11} (cf.\ \cite[Theorem~4.1]{GLMST11}):

\begin{theorem} \lb{t1.1} Let $\cH$ be a separable Hilbert space and 
$T \geq 0$ be a self-adjoint operator with $T^{-1}\in\cB(\cH)$. Assume that
$S$ is closed and densely defined in $\cH$, with
$\big(\dom(S)\cap\dom(S^*)\big) \supseteq \dom(T)$, implying
$ST^{-1}\in\cB(\cH)$ and $S^*T^{-1}\in\cB(\cH)$. If, in addition, 
$ST^{-1}\in\cB_1(\cH)$ and $S^*T^{-1}\in\cB_1(\cH)$, then 
\begin{equation}
T^{-1/2}ST^{-1/2} \in\cB_1(\cH), \quad 
(T^{-1/2}ST^{-1/2})^* = T^{-1/2}S^*T^{-1/2} \in\cB_1(\cH).    \lb{1.1}
\end{equation}
Moreover,  
\begin{equation} \lb{1.2}
\big\|T^{-1/2}ST^{-1/2}\big\|_{\cB_1(\cH)} 
= \big\|T^{-1/2}S^*T^{-1/2}\big\|_{\cB_1(\cH)}
\le \big\|ST^{-1}\big\|_{\cB_1(\cH)}^{1/2} \, \big\|S^*T^{-1}\big\|_{\cB_1(\cH)}^{1/2}.
\end{equation}
\end{theorem}

Theorem \ref{t1.1} was used repeatedly in \cite{GLMST11} (in Section~4 and especially, 
in Section~6). We then announced the present paper in 2010, but due to a variety of 
reasons, finishing it was delayed for quite a while. We should also mention that in the 
meantime we became aware of a paper by Huang \cite{Hu88}, who proved, in fact, 
extended, some parts of Lesch's Proposition~A.1 in \cite{Le05} already in 1988 (we will 
return to this in Sections \ref{s2} and \ref{s3}). 

Given Theorem \ref{t1.1}, we became interested in extensions of it of the following 
three types: \\[1mm]  
\indent 
$\bullet$ The case of fractional powers of $T$ different from $1/2$. \\[1mm] 
\indent 
$\bullet$ General trace ideals $\cB_p(\cH)$, $p \in (1,\infty)$.  \\[1mm] 
\indent 
$\bullet$ Classes of non-self-adjoint operators $T$, especially, sectorial operators $T$ having \\
\hspace*{6mm} bounded imaginary powers. \\[1mm] 
\indent 
While interpolation theory has long been raised to a high art, we emphasize that the methods 
we use are entirely elementary, being grounded in complex interpolation, particularly, 
in Hadamard's three-lines theorem as pioneered by Kato \cite[Sect.~3]{Ka62}, and Heinz's inequality.
In fact, Kato \cite[Sect.~3]{Ka62} presents a proof of the generalized Heinz inequality applying Hadamard's three-lines theorem, and hence the latter is the ultimate ingredient in our proofs. 

We continue with a brief summary of the content of each section. 
One of the principal results proven in Section \ref{s2} reads as follows: 
Assume that $T_j$ are self-adjoint operators in $\cH_j$ with $T_j^{-1} \in \cB(\cH_j)$, 
$j=1,2$, and suppose that $S$ is a closed operator mapping $\dom(S) \subseteq \cH_1$ 
into $\cH_2$ satisfying $\dom(S) \supseteq \dom(T_1)$ and $\dom(S^*) \supseteq \dom(T_2)$.    Then $T_2^{-z}ST_1^{-1+z}$  defined on $\dom(T_1)$, $z\in\ol\Sigma$, is closable, and 
given $ k \in (0,\infty)$, one obtains
\begin{align}
& \big\|\ol{T_2^{-z}ST_1^{-1+z}}\big\|_{\cB(\cH_1,\cH_2)}
\leq \big\|ST_1^{-1}\big\|_{\cB(\cH_1,\cH_2)}^{1-\Re(z)} \,
\big\|S^*T_2^{-1}\big\|_{\cB(\cH_2,\cH_1)}^{\Re(z)}    \no \\
& \quad \times \begin{cases} e^{k |\Im(z)|^2 + k \Re(z)[1 - \Re(z)] + k^{-1} \pi^2}, \\
e^{k |\Im(z)|^2 + k \Re(z)[1 - \Re(z)] + (4k)^{-1} \pi^2}, \text{ if $T_1\geq 0$, or $T_2\geq 0$,} \\
1, \text{ if $T_j\geq 0$, $j=1,2$,}
\end{cases} \quad z \in \ol \Sigma,     \lb{1.3}
\end{align}
as well as 
\begin{align}
& \big\|\ol{T_2^{-x}ST_1^{-1+x}}\big\|_{\cB(\cH_1,\cH_2)}
\leq \big\|ST_1^{-1}\big\|_{\cB(\cH_1,\cH_2)}^{1- x} \,
\big\|S^*T_2^{-1}\big\|_{\cB(\cH_2,\cH_1)}^x    \no \\
& \quad \times \begin{cases} e^{2 \pi [x(1-x)]^{1/2}}, \\
e^{\pi [x(1-x)]^{1/2}}, \text{ if $T_1\geq 0$, or $T_2\geq 0$,} \\
1, \text{ if $T_j\geq 0$, $j=1,2$,}
\end{cases} \quad x \in [0,1].     \lb{1.4}
\end{align}

In Section \ref{s3} we turn to trace ideals $\cB_p(\cH)$, $p \in (1,\infty)$. In addition to the hypotheses imposed on $T_j$, $j=1,2$, and $S$ mentioned in the paragraph preceding 
\eqref{1.3}, let $p\in [1,\infty)$ and $ST_1^{-1} \in \cB_p(\cH_1,\cH_2)$ and 
$S^*T_2^{-1} \in \cB_p(\cH_2,\cH_1)$. Then given $ k \in (0,\infty)$, the principal result in 
Section \ref{s3} derives the analog of \eqref{1.3} and \eqref{1.4} in the form, 
\begin{align}
& \big\|\ol{T_2^{-z}ST_1^{-1+z}}\big\|_{\cB_p(\cH_1,\cH_2)}
\leq \big\|ST_1^{-1}\big\|_{\cB_p(\cH_1,\cH_2)}^{1-\Re(z)} \,
\big\|S^*T_2^{-1}\big\|_{\cB_p(\cH_2,\cH_1)}^{\Re(z)}    \no \\
& \quad \times \begin{cases} e^{k |\Im(z)|^2 + k \Re(z)[1 - \Re(z)] + k^{-1} \pi^2}, \\
e^{k |\Im(z)|^2 + k \Re(z)[1 - \Re(z)] + (4k)^{-1} \pi^2}, \text{ if $T_1\geq 0$, or $T_2\geq 0$,} \\
1, \text{ if $T_j\geq 0$, $j=1,2$,}
\end{cases} \quad z \in \ol \Sigma,    \lb{1.5}
\end{align}
as well as 
\begin{align}
& \big\|\ol{T_2^{-x}ST_1^{-1+x}}\big\|_{\cB_p(\cH_1,\cH_2)}
\leq \big\|ST_1^{-1}\big\|_{\cB_p(\cH_1,\cH_2)}^{1- x} \,
\big\|S^*T_2^{-1}\big\|_{\cB_p(\cH_2,\cH_1)}^x    \no \\
& \quad \times \begin{cases} e^{2 \pi [x(1-x)]^{1/2}}, \\
e^{\pi [x(1-x)]^{1/2}}, \text{ if $T_1\geq 0$, or $T_2\geq 0$,} \\
1, \text{ if $T_j\geq 0$, $j=1,2$,}
\end{cases} \quad x \in [0,1].    \lb{1.6}
\end{align}

In our final Section \ref{s4} we discuss the extension of \eqref{1.3} and \eqref{1.4} 
from self-adjoint to sectorial operators $T_j$, $j=1,2$. One of our principal results there 
reads as follows: Assume that $T_j$ are  sectorial operators in $\cH_j$ such 
that $T_j^{-1} \in \cB(\cH_j)$, and that for some $\theta_j \geq 0$, $N_j \geq 1$, 
$\big\|T^{is}_j\big\|_{\cB(\cH_j)} \leq N_j e^{\theta_j |s|}$, $s \in \bbR$, 
$j=1,2$. In addition, suppose that $S$ is a closed operator mapping 
$\dom(S) \subseteq \cH_1$ into $\cH_2$, satisfying $\dom(S) \supseteq \dom(T_1)$ and 
$\dom(S^*) \supseteq \dom(T_2^*)$. Then $T_2^{-z}ST_1^{-1+z}$ defined on 
$\dom(T_1)$, $z\in\ol\Sigma$, is closable, and given $k \in (0, \infty)$, one obtains
\begin{align}
& \big\|\ol{T_2^{-z}ST_1^{-1+z}}\big\|_{\cB(\cH_1,\cH_2)} \leq N_1 N_2
e^{k (\Im(z))^2 + k \Re(z) [1-\Re(z)] + (4k)^{-1} (\theta_1 + \theta_2)^2}   \no \\
& \quad \times \big\|ST_1^{-1}\big\|_{\cB(\cH_1,\cH_2)}^{1-\Re(z)}
\, \big\|S^*(T_2^*)^{-1}\big\|_{\cB(\cH_2,\cH_1)}^{\Re(z)}, \quad
z \in \ol \Sigma,    \lb{1.7}
\end{align}
as well as 
\begin{align}
\begin{split}
& \big\|\ol{T_2^{-x}ST_1^{-1+x}}\big\|_{\cB(\cH_1,\cH_2)}
\leq N_1  N_2  e^{(\theta_1  + \theta_2) [x(1-x)]^{1/2}}     \\
& \quad \times \big\|ST_1^{-1}\big\|_{\cB(\cH_1,\cH_2)}^{1-x} \,
\big\|S^*(T_2^*)^{-1}\big\|_{\cB(\cH_2,\cH_1)}^{x}, \quad x \in
[0,1].    \lb{1.8}
\end{split}
\end{align}
Moreover, in addition to the hypotheses on $T_j$, $j=1,2$, and $S$ mentioned in the  
paragraph preceding \eqref{1.7}, let $p\in [1,\infty)$ and $ST_1^{-1} \in \cB_p(\cH_1,\cH_2)$, 
$S^*(T_2^*)^{-1} \in \cB_p(\cH_2,\cH_1)$. Then given $k \in (0, \infty)$, one 
obtains the analog of \eqref{1.5} and \eqref{1.6} in the form, 
\begin{align}
& \big\|\ol{T_2^{-z}ST_1^{-1+z}}\big\|_{\cB_p(\cH_1,\cH_2)} \leq
N_1 N_2
e^{k (\Im(z))^2 + k \Re(z) [1-\Re(z)] + (4k)^{-1} (\theta_1  + \theta_2)^2}   \no \\
& \quad \times
\big\|ST_1^{-1}\big\|_{\cB_p(\cH_1,\cH_2)}^{1-\Re(z)} \,
\big\|S^*(T_2^*)^{-1}\big\|_{\cB_p(\cH_2,\cH_1)}^{\Re(z)}, \quad z
\in \ol \Sigma,    \lb{1.9}
\end{align}
as well as 
\begin{align}
\begin{split}
& \big\|\ol{T_2^{-x}ST_1^{-1+x}}\big\|_{\cB_p(\cH_1,\cH_2)}
\leq N_1 N_2  e^{(\theta_1 + \theta_2) [x(1-x)]^{1/2}}   \\
& \quad \times \big\|ST_1^{-1}\big\|_{\cB_p(\cH_1,\cH_2)}^{1-x} \,
\big\|S^*(T_2^*)^{-1}\big\|_{\cB_p(\cH_2,\cH_1)}^{x},  \quad x \in
[0,1].    \lb{1.10}
\end{split}
\end{align}

We note that we permit operators $T$ in \eqref{1.3}--\eqref{1.6} to have spectrum covering 
$\bbR$ except for a neighborhood of zero. Thus, our results in Section \ref{s4} for sectorial 
operators $T$ do not cover the results  \eqref{1.3}--\eqref{1.6}. 

In conclusion, we briefly summarize the basic notation used in this paper: Let
$\cH$ be a separable complex Hilbert space, $(\cdot,\cdot)_{\cH}$ the
scalar product in $\cH$ (linear in the second factor), and $I_{\cH}$ the identity operator 
in $\cH$. Limits in the norm topology on $\cH$ (also called strong limits in $\cH$) will be 
denoted by $\slim$. If $T$ is a linear operator mapping (a subspace of\,) a
Hilbert space into another, $\dom(T)$ denotes the domain of $T$. The closure
of a closable operator $S$ is denoted by $\ol S$. The spectrum and
resolvent set of a closed linear operator in $\cH$ will be denoted by
$\sigma(\cdot)$  and $\rho(\cdot)$, respectively. The
Banach spaces of bounded and compact linear operators in $\cH$ are
denoted by $\cB(\cH)$ and $\cB_\infty(\cH)$, respectively; in the context of two
Hilbert spaces, $\cH_j$, $j=1,2$, we use the analogous abbreviations
$\cB(\cH_1, \cH_2)$ and $\cB_\infty(\cH_1, \cH_2)$. Similarly,
the usual $\ell^p$-based Schatten--von Neumann (trace) ideals are denoted
by $\cB_p(\cH)$, $p\in [1,\infty)$.

\section{Interpolation and some Operator Norm Bounds  Revisited}  \lb{s2}

In this section we revisit and extend a number of bounds
collected by Lesch in \cite[Proposition\ A.1]{Le05}:

Through most of this section we will make the following assumptions:

\begin{hypothesis} \lb{h2.1}
Assume that $T$ is a self-adjoint operator in $\cH$ with $T^{-1} \in \cB(\cH)$.
In addition, suppose that $S$ is a closed operator in $\cH$ satisfying
\begin{equation}
\dom(S) \cap \dom(S^*) \supseteq \dom(T).   \lb{2.-1}
\end{equation}
\end{hypothesis}

In particular, Hypothesis \ref{h2.1} implies that
\begin{equation}
ST^{-1} \in \cB(\cH), \quad S^*T^{-1} \in \cB(\cH).     \lb{2.0}
\end{equation}

\begin{remark} \lb{r2.2}
$(i)$ In the sequel we will adhere to the following convention: Operator
products $AB$ of two linear operators $A$ and $B$ in $\cH$ are always
assumed to be maximally defined, that is,
\begin{equation}
\dom(AB) = \{f\in\cH \,|\, f\in\dom(B), \, Bf \in \dom(A)\},    \lb{2.0a}
\end{equation}
unless explicitly stated otherwise. The same convention is of course applied to
products of three or more linear operators in $\cH$. \\
$(ii)$ We recall the following useful facts (see, e.g., \cite[Theorem\ 4.19 ]{We80}): Suppose $T_j$, $j=1,2$, are two densely defined linear operators in $\cH$ such that $T_2 T_1$ is also densely defined in
$\cH$. Then,
\begin{equation}
 (T_2 T_1)^* \supseteq T_1^* T_2^*.     \lb{2.0b}
 \end{equation}
If in addition $T_2 \in \cB(\cH)$, then
\begin{equation}
(T_2 T_1)^* = T_1^* T_2^*.      \lb{2.0c}
\end{equation}
\end{remark}

\begin{theorem}   \lb{t2.3}
Assume Hypothesis \ref{h2.1}. Then the following facts hold: \\
$(i)$ The operators $T^{-1}ST$ and $T^{-1}S^*T$ are well-defined on
$\dom(T^2)$, and hence densely defined in $\cH$,
\begin{equation}
\dom\big(T^{-1}ST\big) \cap \dom\big(T^{-1}S^*T\big) \supseteq \dom\big(T^2\big).     \lb{2.1}
\end{equation}
$(ii)$ The relations
\begin{equation}
\big(T^{-1}ST\big)^* = TS^*T^{-1}, \quad \big(T^{-1}S^*T\big)^* = TST^{-1},   \lb{2.2}
\end{equation}
hold, and hence $TS^*T^{-1}$ and $TST^{-1}$ are closed in $\cH$. \\
$(iii)$ One infers that
\begin{equation}
T^{-1}ST \, \text{ is bounded  if and only if } \,
(T^{-1}ST)^* = TS^*T^{-1} \in \cB(\cH).    \lb{2.3}
\end{equation}
In case $T^{-1}ST$ is bounded, then
\begin{equation}
\ol{T^{-1}ST} = (TS^*T^{-1})^*, \quad
\big\|\ol{T^{-1}ST}\big\|_{\cB(\cH)} = \big\|TS^*T^{-1}\big\|_{\cB(\cH)}.    \lb{2.4}
\end{equation}
Analogously, one concludes that
\begin{equation}
T^{-1}S^*T \, \text{ is bounded  if and only if } \,
\big(T^{-1}S^*T\big)^* = TST^{-1} \in \cB(\cH).    \lb{2.5}
\end{equation}
In case $T^{-1}S^*T$ is bounded, then
\begin{equation}
\ol{T^{-1}S^*T} = \big(TST^{-1}\big)^*, \quad
\big\|\ol{T^{-1}S^*T}\big\|_{\cB(\cH)} = \big\|TST^{-1}\big\|_{\cB(\cH)}.    \lb{2.6}
\end{equation}
\end{theorem}
\begin{proof}
We start by recalling that
\begin{align}
\dom\big(T^{-1}ST\big) &= \big\{g \in \dom(T) \,\big|\, Tg \in \dom(S)\big\},    \lb{2.7} \\
\dom\big(TS^*T^{-1}\big) &= \big\{f \in \cH \,\big|\, S^*T^{-1}f \in \dom(T)\big\}.   \lb{2.8}
\end{align}
$(i)$ Suppose that $g \in \dom\big(T^2\big)$. Then
$Tg \in \dom(T) \subseteq \dom(S)$ and hence
\begin{equation}
\dom\big(T^{-1}ST\big) \supseteq \dom\big(T^2\big).    \lb{2.9}
\end{equation}
Since $T$ and hence $T^2$ are self-adjoint in $\cH$ and hence
necessarily densely defined, $T^{-1}ST$ is densely defined in $\cH$. The same applies to
$T^{-1} S^* T$. \\
$(ii)$ Applying Remark \ref{r2.2}\,$(ii)$, one obtains
\begin{equation}
\big(T^{-1}ST\big)^* = \big(T^{-1}[ST]\big)^* = [ST]^* T^{-1}\supseteq T S^* T^{-1},  \lb{2.11}
\end{equation}
since $T^{-1}ST$ is densely defined in $\cH$ by item $(i)$. To prove the converse inclusion in \eqref{2.11}, we now assume that
$f\in \dom\big(\big(T^{-1}ST\big)^*\big)$ and $g\in \dom\big(T^2\big) \subseteq \dom\big(T^{-1}ST\big)$. Then
\begin{equation}
\big(\big(T^{-1}ST\big)^*f,g\big)_{\cH}
= \big(f,T^{-1}ST g\big)_{\cH} = \big(T^{-1}f,STg\big)_{\cH} = \big(S^*T^{-1}f, Tg\big)_{\cH}.    \lb{2.12}
\end{equation}
Since $\dom\big(T^2\big)$ is an operator core for $T$, \eqref{2.12} extends to all
$g \in \dom(T)$, that is, one has
\begin{equation}
\big(\big(T^{-1}ST\big)^*f,g\big)_{\cH} = \big(S^*T^{-1}f, Tg\big)_{\cH}, \quad
f\in \dom\big(\big(T^{-1}ST\big)^*\big), \; g\in \dom(T).    \lb{2.13}
\end{equation}
Consequently, $S^*T^{-1}f \in \dom(T)$ and
\begin{align}
\begin{split}
\big(\big(T^{-1}ST\big)^*f,g\big)_{\cH} = \big(S^*T^{-1}f,Tg\big)_{\cH} = \big(TS^*T^{-1}f,g\big)_{\cH}&,
\lb{2.14} \\
f\in \dom\big(\big(T^{-1}ST\big)^*\big), \; g\in \dom(T)&,
\end{split}
\end{align}
implying
\begin{equation}
\big(T^{-1}ST\big)^*f = TS^*T^{-1}f, \quad f\in \dom\big(\big(T^{-1}ST\big)^*\big),    \lb{2.15}
\end{equation}
and hence,
\begin{equation}
\big(T^{-1}ST\big)^* \subseteq TS^*T^{-1}.    \lb{2.16}
\end{equation}
Then \eqref{2.11} and \eqref{2.16} yield the first relation in \eqref{2.2}. Replacing $S$
by $S^*$ yields the second relation in \eqref{2.2}.
\\
$(iii)$ Since $T^{-1}ST$ is densely defined by \eqref{2.1}, an application of
\cite[Theorem 4.14(a)]{We80} yields \eqref{2.3}. Equation \eqref{2.4}
is an immediate consequence of \eqref{2.3}.

Again, replacing $S$ by $S^*$ implies \eqref{2.5} and \eqref{2.6}.
\end{proof}

In the following we denote by $\Sigma \subset \bbC$ the open strip
\begin{equation}
\Sigma = \{z\in\bbC \,|\, \Re(z) \in (0,1)\},    \lb{2.16a}
\end{equation}
and by $\ol \Sigma$ its closure.

To state additional results we will have to apply a version of Hadamard's
three-lines theorem and hence recall the following general result:

\begin{theorem} [\cite{Hi52} (see also \cite{GK69}, Sect.\ III.13)] \lb{t2.4}
Suppose $\phi(\cdot)$ is an analytic function on $\Sigma$, continuous on $\ol \Sigma$, and
satisfying for some fixed $C\in\bbR$ and $a \in [0,\pi)$,
\begin{equation}
\sup_{z \in \ol \Sigma} \Big[e^{-a |\Im(z)|} \ln(|\phi(z)|) \Big] \leq C.    \lb{2.17}
\end{equation}
Then
\begin{align}
& |\phi(z)|  \leq \exp\bigg\{\f{\sin(\pi \Re(z))}{2} \int_{\bbR} dy \,
\bigg[\f{\ln(\phi(iy))}{\cosh(\pi (y-\Im(z)))-\cosh(\pi \Re(z))}   \no \\
& \hspace*{4.8cm} +
\f{\ln(\phi(1+iy))}{\cosh(\pi (y-\Im(z)))+\cosh(\pi \Re(z))}\bigg]\bigg\},   \no \\
& \hspace*{9.7cm}   z \in \Sigma.    \lb{2.18}
\end{align}
If in addition, for some $C_0, C_1 \in (0,\infty)$,
\begin{equation}
|\phi(iy)| \leq C_0, \quad |\phi(1+iy)| \leq C_1, \quad y\in\bbR,     \lb{2.19}
\end{equation}
then
\begin{equation}
|\phi(z)| \leq C_0^{1-\Re(z)} C_1^{\Re(z)}, \quad z\in{\ol \Sigma}.   \lb{2.21}
\end{equation}
\end{theorem}

For a recent detailed exposition of such results we refer to
\cite[Sects.\ 1.3.2, 1.3.3]{Gr08}. A classical application of Theorem\ \ref{t2.4} to linear 
operators appeared in \cite{St56} (see also \cite[Sect.~4.3]{BS88}).   

The growth condition \eqref{2.17} is of course familiar from
Phragmen--Lindel\"of-type arguments applied to the strip $\Sigma$ (see, e.g.,
\cite[Theorem\ 12.9]{Ru87}).

\medskip

In the sequel, complex powers $T^z$, $z\in\ \ol \Sigma$, of a self-adjoint operator $T$
in $\cH$, with $T^{-1} \in \cB(\cH)$, are defined in terms of the spectral representation of $T$,
\begin{equation}
T = \int_{\sigma(T)} \lambda \, dE_T(\lambda),
\end{equation}
with $\{E_T(\lambda)\}_{\lambda \in \bbR}$
denoting the family of spectral projections of $T$, as follows: Since by hypothesis,
$(-\varepsilon, \varepsilon) \cap \sigma(T) = \emptyset$ for some
$\varepsilon > 0$, one defines
\begin{equation}
T^z = \int_{\sigma(T)} \lambda^z \, dE_T(\lambda), \quad z \in \ol \Sigma,
\end{equation}
where
\begin{equation}
\lambda^z = \lambda^{\Re(z)} e^{i \Im(z) \ln(|\lambda|)} \big[\theta(\lambda) + e^{- \pi \Im(z)}
\theta (-\lambda) \big], \quad \lambda \in \bbR\backslash\{0\}, \; z \in \ol \Sigma,
\end{equation}
and
\begin{equation}
\theta(x) = \begin{cases} 1, & x > 0, \\ 0, & x < 0. \end{cases}
\end{equation}
Consequently, one obtains the estimate
\begin{equation}
\big\|T^{iy}\big\|_{\cB(\cH)} \leq \max \big(1, e^{- \pi y}\big) \leq e^{\pi |y|}, \quad y \in \bbR,  \lb{2.23}
\end{equation}
and
\begin{equation}
\text{if $T \geq \varepsilon I_{\cH}$, for some $\varepsilon > 0$, then $T^{iy}$ is unitary, } \, 
\big\|T^{iy}\big\|_{\cB(\cH)} = 1, \quad y \in \bbR.    \lb{2.24}
\end{equation}

\begin{theorem}   \lb{t2.5}
Assume Hypothesis \ref{h2.1} and suppose that $TST^{-1}\in\cB(\cH)$ as well as
$TS^*T^{-1}\in\cB(\cH)$. Then $S\in\cB(\cH)$ $($and hence $S^* \in\cB(\cH)$$)$ and
\begin{equation}
\|S \|_{\cB(\cH)} = \|S^* \|_{\cB(\cH)}
\leq \big\|TST^{-1}\big\|_{\cB(\cH)}^{1/2} \, \big\|TS^*T^{-1}\big\|_{\cB(\cH)}^{1/2}
\begin{cases} e^{2 \pi}, \\ 1, \text{ if $T \geq 0$.}
\end{cases}     \lb{2.25}
\end{equation}
\end{theorem}
\begin{proof}
Introducing
\begin{equation}
\varphi_k (z) = e^{kz(z-1)}\big(T^2 f, T^{2z-3}ST^{-1-2z} T^2 g\big)_{\cH},
\quad f, g \in \dom\big(T^2\big), \; z \in \ol \Sigma, \; k \in [0,\infty),   \lb{2.26}
\end{equation}
one infers that $\varphi_k$ is analytic on $\Sigma$. (We note that the idea to exploit the factor
$e^{kz(z-1)}$, $k>0$, can already be found in the proof of \cite[Theorem\ 6]{Ka62}. This factor is
used in \eqref{2.27}--\eqref{2.28a} below to neutralize factors of the type $e^{4\pi |y|}$ and
$e^{4\pi |\Im(z)|}$.)

In the following we focus on the general case where $T$ is self-adjoint and $k>0$; in this case we will employ the bound \eqref{2.23}.

Assuming $k > 0$, \eqref{2.23} yields the estimates
\begin{align}
|\varphi_k (iy)| &= e^{- k y^2} \big|\big(f, T^{2iy-1}ST^{1-2iy}g\big)_{\cH}\big|
= e^{- k y^2} \big|\big(T^{-2iy}f, [T^{-1}ST] T^{-2iy}g\big)_{\cH}\big|    \no \\
& \leq e^{- k y^2 + 4 \pi |y|} \big\|\ol{T^{-1}ST}\big\|_{\cB(\cH)} \, \|f\|_{\cH} \, \|g\|_{\cH}    \no \\
& = e^{- k y^2 + 4 \pi |y|} \big\|TS^*T^{-1}\big\|_{\cB(\cH)} \, \|f\|_{\cH} \, \|g\|_{\cH},    \no \\
& \leq e^{k^{-1} 4 \pi^2} \big\|TS^*T^{-1}\big\|_{\cB(\cH)} \, \|f\|_{\cH} \, \|g\|_{\cH},
\quad f, g \in \dom\big(T^2\big), \; y\in\bbR,  \lb{2.27}
\end{align}
using \eqref{2.4}, and similarly,
\begin{align}
& |\varphi_k (1+iy)| =  e^{- k y^2} \big|\big(f, T^{1+2iy}ST^{-1-2iy}g\big)_{\cH}\big|    \no \\
& \quad =  e^{- k y^2} \big|\big(T^{-2iy}f, [TST^{-1}] T^{-2iy}g\big)_{\cH}\big|
\no \\
& \quad \leq e^{k^{-1} 4 \pi^2} \big\|TST^{-1}\big\|_{\cB(\cH)} \, \|f\|_{\cH} \, \|g\|_{\cH}, \quad
 f, g \in \dom\big(T^2\big), \; y\in\bbR.     \lb{2.28}
\end{align}
In addition, one obtains
\begin{align}
|\varphi_k (z)| &= e^{- k |\Im(z)|^2 + k \Re(z)[\Re(z) - 1]}
\big|\big(T^2 f,T^{2z-3} S T^{-1-2z} T^2 g\big)_{\cH}\big|  \no \\
& \leq e^{- k |\Im(z)|^2 + k \Re(z)[\Re(z) - 1] + 4 \pi |\Im(z)|}     \no \\
& \quad \times \big\|T^{2\Re(z)-3}\big\|_{\cB(\cH)} \,
\big\|ST^{-1}\big\|_{\cB(\cH)} \,
\big\|T^{-2\Re(z)}\big\|_{\cB(\cH)} \, \big\|T^2 f\big\|_{\cH} \, \big\|T^2 g\big\|_{\cH}   \no \\
& \leq e^{k^{-1}4\pi^2 + k \Re(z)[\Re(z) - 1]}      \no \\
& \quad \times \big\|T^{2\Re(z)-3}\big\|_{\cB(\cH)} \,
\big\|ST^{-1}\big\|_{\cB(\cH)} \,
\big\|T^{-2\Re(z)}\big\|_{\cB(\cH)} \, \big\|T^2 f\big\|_{\cH} \, \big\|T^2 g\big\|_{\cH}   \no \\
&\leq C, \quad f, g \in \dom\big(T^2\big), \; z\in \ol \Sigma,
\lb{2.28a}
\end{align} 
for some finite constant $C=C(f,g,S,T)>0$, independent of $z\in \ol \Sigma$.

Applying the Hadamard three-lines estimate \eqref{2.21} to $\varphi(\cdot)$ then yields
\begin{align}
\begin{split}
|\varphi_k (z)| \leq e^{k^{-1} 4 \pi^2} \big\|TS^*T^{-1}\big\|_{\cB(\cH)}^{1-\Re(z)} \, \big\|TST^{-1}\big\|_{\cB(\cH)}^{\Re(z)}
\, \|f\|_{\cH} \, \|g\|_{\cH},&  \\
f, g \in \dom\big(T^2\big), \; z\in\ol \Sigma.&  \lb{2.29}
\end{split}
\end{align}
Taking $z=1/2$ in \eqref{2.29} implies
\begin{align}
\begin{split} 
|(f, Sg)_{\cH}| \leq e^{4^{-1} k + k^{-1} 4 \pi^2} \big\|TS^*T^{-1}\big\|_{\cB(\cH)}^{1/2} \,
\big\|TST^{-1}\big\|_{\cB(\cH)}^{1/2} \, \|f\|_{\cH} \, \|g\|_{\cH},&   \\
f, g \in \dom\big(T^2\big).&  \lb{2.29A}
\end{split} 
\end{align}
Optimizing with respect to $k > 0$ yields 
\begin{equation}
|(f, Sg)_{\cH}| \leq e^{2 \pi} \big\|TS^*T^{-1}\big\|_{\cB(\cH)}^{1/2} \,
\big\|TST^{-1}\big\|_{\cB(\cH)}^{1/2} \, \|f\|_{\cH} \, \|g\|_{\cH},
\quad f, g \in \dom\big(T^2\big).  \lb{2.29B}
\end{equation} 
Since $\dom\big(T^2\big)$ is dense in $\cH$, this yields that $S$ is a bounded operator in $\cH$.
Employing that $S$ is closed in $\cH$ finally proves $S\in\cB(\cH)$  and hence the first estimate in \eqref{2.25}.

If in addition, $T \geq 0$, we choose $k=0$ in \eqref{2.26} and then rely on equality \eqref{2.24} (as opposed to \eqref{2.23}),
which slightly simplifies the estimates \eqref{2.27}--\eqref{2.29B}, implying the second inequality in \eqref{2.25}.
\end{proof}
\begin{remark} \lb{r2.5a}
In the special case where $T$ is self-adjoint and $T \geq \varepsilon I_{\cH}$ for some 
$\varepsilon > 0$, there exists an alternative way of deriving the bound \eqref{2.25} by means of Proposition~A.1\,(2) proved by Lesch \cite{Le05} in the context of closed, symmetric operators $S$. 
Indeed, an application of \cite[Proposition~A.1\,(2)]{Le05} with $S$ replaced by the symmetric, in fact, self-adjoint, $S^* S$ yields 
\begin{align}
\begin{split}  
\|S^* S \|_{\cB(\cH)} & \leq \big\|T S^* S T^{-1}\big\|_{\cB(\cH)} 
= \big\|\big(T S^* T^{-1}\big) \big(T S T^{-1}\big)\big\|_{\cB(\cH)}    \\ 
& \leq \big\|T S^* T^{-1}\big\|_{\cB(\cH)} \big\|T S T^{-1}\big\|_{\cB(\cH)}. 
\end{split} 
\end{align} 
Thus, $S^* S$ is bounded, so both $S$ and $S^*$ are bounded, and hence, 
\begin{equation} 
\|S\|_{\cB(\cH)} = \|S^*\|_{\cB(\cH)} = \| (S^* S)\|^{1/2}_{\cB(\cH)} 
\leq \big\|T S T^{-1}\big\|^{1/2}_{\cB(\cH)} \big\|T S^* T^{-1}\big\|^{1/2}_{\cB(\cH)}.
\end{equation} 
Thus, in this special case one needs no additional arguments to prove Theorem \ref{t2.5}. However,
this type of argument does not apply to the remaining statements in this section.
\end{remark}
Theorem \ref{t2.5} allows us to derive the following result.
\begin{theorem}  \lb{t2.6}
In addition to Hypothesis \ref{h2.1} suppose that $T \geq 0$. Then
\begin{equation}
T^{-1/2}ST^{-1/2}\in\cB(\cH), \quad T^{-1/2}S^*T^{-1/2}\in\cB(\cH),    \lb{2.29a}
\end{equation}
and
\begin{equation}
\big(T^{-1/2}ST^{-1/2}\big)^* = T^{-1/2}S^*T^{-1/2}, \quad
\big(T^{-1/2}S^*T^{-1/2}\big)^* = T^{-1/2}ST^{-1/2}.      \lb{2.30}
\end{equation}
Moreover,
\begin{equation}
\big\|T^{-1/2}ST^{-1/2}\big\|_{\cB(\cH)} = \big\|T^{-1/2}S^*T^{-1/2} \big\|_{\cB(\cH)}
\leq \big\|ST^{-1}\big\|_{\cB(\cH)}^{1/2} \, \big\|S^*T^{-1}\big\|_{\cB(\cH)}^{1/2}.    \lb{2.31}
\end{equation}
\end{theorem}
\begin{proof}
The combined assumptions on $T$ actually yield $T \geq \varepsilon I_{\cH}$ for some
$\varepsilon > 0$. (The condition $T \geq 0$ has inadvertently been omitted in
\cite[Theorem\ 4.1]{GLMST11} and \cite[Proposition~A.1\,(3)]{Le05}.) Introduce the operators
$\hatt S$ and $\hatt T$ in $\cH$ by
\begin{equation}
\hatt S = T^{-1/2} S T^{-1/2}, \quad \hatt T = T^{1/2}.    \lb{2.32}
\end{equation}
Then $\hatt T$ is self-adjoint and
$\dom\big(\hatt S\big) \supseteq \dom\big(T^{1/2}\big)$, that is,
$\dom\big(\hatt S\big) \supseteq \dom\big(\hatt T\big)$ yields that $\hatt S$
is densely defined. Next, we note that by Remark \ref{r2.2}\,$(ii)$,
\begin{equation}
\big(\hatt S\big)^* = \big(T^{-1/2} \big[S T^{-1/2}\big]\big)^* =
\big[S T^{-1/2}\big]^* T^{-1/2} \supseteq T^{-1/2} S^* T^{-1/2},    \lb{2.33}
\end{equation}
and hence $\dom\big(\big(\hatt S\big)^*\big) \supseteq \dom\big(\hatt T\big)$.
Moreover, since
\begin{equation}
\hatt T \hatt S \big(\hatt T\big)^{-1} = S T^{-1} \in \cB(\cH), \quad
\hatt T \big(\hatt S\big)^* \big(\hatt T\big)^{-1} \supseteq S^* T^{-1} \in \cB(\cH),
\lb{2.34}
\end{equation}
by Hypothesis \ref{h2.1} (resp., \eqref{2.0}), one also infers
\begin{equation}
\hatt T \big(\hatt S\big)^* \big(\hatt T\big)^{-1} = S^* T^{-1} \in \cB(\cH).
\lb{2.35}
\end{equation}
Thus, Theorem \ref{t2.5} applies to $\hatt S$ and $\hatt T$ and hence
$\hatt S, \, \big(\hatt S\big)^* \in \cB(\cH)$, as well as
\begin{align}
\big\|\hatt S\big\|_{\cB(\cH)} = \big\|\big(\hatt S\big)^*\big\|_{\cB(\cH)}
& \leq \big\|\hatt T \hatt S \big(\hatt T\big)^{-1}\big\|_{\cB(\cH)}^{1/2} \,
\big\|\hatt T \big(\hatt S\big)^* \big(\hatt T\big)^{-1}\big\|_{\cB(\cH)}^{1/2}   \no \\
& = \big\|S T^{-1}\big\|_{\cB(\cH)}^{1/2} \, \big\|S^* T^{-1}\big\|_{\cB(\cH)}^{1/2}.
\lb{2.36}
\end{align}
Since our hypotheses are symmetric with respect to $S$ and $S^*$,
interchanging $S$ and $S^*$, repeating \eqref{2.32}--\eqref{2.36} with $\hatt S$
replaced by
\begin{equation}
\wti S = T^{-1/2} S^* T^{-1/2},   \lb{2.37}
\end{equation}
then also yields $\big(\wti S\big)^* \supseteq T^{-1/2} S T^{-1/2} = \hatt S$. Since
$\hatt S \in\cB(\cH)$, one concludes that $\big(\wti S\big)^* = \hatt S$. Applying
Theorem \ref{t2.5} to $\wti S$ and $\hatt T$ then yields $\wti S \in \cB(\cH)$ and hence 
also $\big(\hatt S\big)^* = \wti S$, completing the proof.
\end{proof}

In the special case where $S$ is symmetric in $\cH$, $S \subseteq S^*$, and $T \geq 0$ 
(actually, $T \geq \varepsilon I_{\cH}$ for some $\varepsilon > 0$ as also the condition 
$T^{-1} \in \cB(\cH)$ is involved), 
nearly all the results of this section up to now (as well as the basic strategy of proofs employed), appeared in Lesch \cite[Appendix A]{Le05}. We emphasize, however, that some of these results, especially, Theorem \ref{t2.6}, were previously derived in 1988 by Huang 
\cite[Lemma~2.1.(b)]{Hu88}. In fact, combining the spectral theorem for $T$ and the three-lines theorem, Huang arrives at an extension of Theorem \ref{t2.6} involving fractional powers of $T^{\alpha}$ $\alpha \in [1/2,1]$ on the right-hand side of \eqref{2.31}. 
  
Next, we recall the generalized polar decomposition for densely defined closed operators $S$ in $\cH$ derived in \cite{GMMN09},
\begin{equation}
S = |S^*|^\alpha U_S |S|^{1-\alpha}, \quad \alpha \in [0,1],   \lb{2.41}
\end{equation}
where $U_S$ denotes the partial isometry in $\cH$ associated with the standard
polar decomposition $S = U_S |S|$, and $|S|=(S^*S)^{1/2}$ (and we interpret
$|S|^0 = I_{\cH}$ in this particular context).

We will employ \eqref{2.41} (and its analog for $S^*$) to prove the following result:

\begin{theorem}   \lb{t2.7}
Assume Hypothesis \ref{h2.1}. Then $T^{-z}ST^{-1+z}$, $z\in\ol\Sigma$, defined on $\dom(T)$,
is closable in $\cH$, and
\begin{align}
\ol{T^{-z}ST^{-1+z}} &= T^{-i \Im(z)} \big[|S^*|^{\Re(z)} T^{- \Re(z)}\big]^*
U_S    \no \\
& \quad \times |S|^{1-\Re(z)} T^{-1+\Re(z)} T^{i \Im(z)} \in \cB(\cH),
\quad z \in \ol \Sigma.    \lb{2.42}
\end{align}
In addition, given $ k \in (0,\infty)$, one obtains
\begin{align}
\big\|\ol{T^{-z}ST^{-1+z}}\big\|_{\cB(\cH)}
& \leq \big\|ST^{-1}\big\|_{\cB(\cH)}^{1-\Re(z)} \, \big\|S^*T^{-1}\big\|_{\cB(\cH)}^{\Re(z)}    \no \\
& \quad \times \begin{cases} e^{k |\Im(z)|^2 + k \Re(z)[1 - \Re(z)] + k^{-1} \pi^2}, \\
1, \text{ if $T\geq 0$,}
\end{cases}  \quad z \in \ol \Sigma,    \lb{2.44}
\end{align}
and
\begin{align}
\big\|\ol{T^{-x}ST^{-1+x}}\big\|_{\cB(\cH)}
& \leq \big\|ST^{-1}\big\|_{\cB(\cH)}^{1- x} \, \big\|S^*T^{-1}\big\|_{\cB(\cH)}^x
\begin{cases} e^{2 \pi [x(1-x)]^{1/2}}, \\
1, \text{ if $T\geq 0$,}
\end{cases} \quad x \in [0,1].    \lb{2.45}
\end{align}
In particular, assuming $T\geq 0$ and taking $x=1/2$ in \eqref{2.45} one recovers the estimate
\eqref{2.31} $($in this particular case the operator closure sign in \eqref{2.44} is superfluous since
$T^{-1/2}ST^{-1/2} \in \cB(\cH)$ by \eqref{2.29a}$)$.
\end{theorem}
\begin{proof}
We start by noting that $\dom(S)\supseteq \dom(T)$ (together with $S$ and $T$ closed by hypothesis) implies that $S$ is relatively bounded with respect to $T$ and hence
there exist $a>0$ and $b>0$ such that
\begin{align}
\||S|f\|_{\cH}^2 = \|Sf\|_{\cH}^2 & \leq a^2 \|Tf\|_{\cH}^2 + b^2 \|f\|_{\cH}^2
= a^2 \||T|f\|_{\cH}^2 + b^2 \|f\|_{\cH}^2  \no \\
& = \big\|\big[a^2 |T|^2 + b^2\big]^{1/2} f\big\|_{\cH}^2,
\quad f \in \dom(T)=\dom(|T|).   \lb{2.46}
\end{align}
Thus, applying Heinz's inequality (cf.\ \cite[Sect.~3.2.1]{Fu02}, \cite[Theorem\ 3]{He51},
\cite{Ka52}, \cite[Theorem\ IV.1.11]{KPS82}, \cite{Mc80}), one infers that
\begin{equation}
\dom\big(|S|^\alpha\big) \supseteq \dom\big(\big(a^2|T|^2 + b^2\big)^{\alpha/2}\big)
= \dom\big(|T|^{\alpha}\big), \quad \alpha \in [0,1],   \lb{2.47}
\end{equation}
and
\begin{equation}
\big\||S|^\alpha \big[a^2 |T|^2 + b^2\big]^{-\alpha/2} h\big\|_{\cH}
\leq \| h \|_{\cH}^2, \quad h \in \cH, \;  \alpha \in [0,1].   \lb{2.48}
\end{equation}
Similarly, interchanging $S$ and $S^*$, one obtains
\begin{equation}
\dom\big(|S^*|^\alpha\big) \supseteq \dom\big(|T|^{\alpha}\big), \quad \alpha \in [0,1],    \lb{2.48a}
\end{equation}
and
\begin{equation}
\big\||S^*|^\alpha \big[\wti a^2 |T|^2 + \wti b^2\big]^{-\alpha/2} h\big\|_{\cH}
\leq \| h\|_{\cH}^2, \quad h \in \cH, \;  \alpha \in [0,1],   \lb{2.48b}
\end{equation}
for appropriate $\wti a >0$, $\wti b >0$.
Applying \eqref{2.41} to $S$ (with $\alpha=\Re(z)$), and using \eqref{2.47} and \eqref{2.48a}, one concludes that
\begin{align}
& T^{-z}ST^{-1+z} = T^{-z} |S^*|^{\Re(z)} U_S |S|^{1-\Re(z)} T^{-1+z}  \no \\
& \quad \subseteq \big[|S^*|^{\Re(z)}
T^{-\ol z}\big]^* U_S \big[|S|^{1-\Re(z)} T^{-1+z}\big]    \lb{2.49} \\
& \quad = T^{-i \Im(z)} \big(|S^*|^{\Re(z)} T^{- \Re(z)}\big)^*
U_S |S|^{1-\Re(z)} T^{-1+\Re(z)} T^{i \Im(z)} \in \cB(\cH),  \quad
z \in \ol \Sigma,  \no
\end{align}
proving \eqref{2.42}.

Next, one defines (repeatedly employing below the fact that for a closable operator $A$, $\ol A$
is an extension of $A$)
\begin{align}
\begin{split}
\phi_k (z) = e^{kz(z-1)} \big(Tf,T^{-1-z} S T^{-2+z} Tg\big)_{\cH}
= e^{kz(z-1)} \big(f, \ol{T^{-z}ST^{-1+z}} g\big)_{\cH},&  \\
 f,g \in \dom(T), \; z \in \ol \Sigma, \; k \in [0,\infty).&   \lb{2.50}
\end{split}
\end{align}

Again we primarily focus on the case where $T$ is merely self-adjoint and hence choose $k > 0$
and employ the estimate \eqref{2.23} in the following.

One estimates
\begin{align}
|\phi(iy)| &= e^{-k y^2} \big|\big(T^{iy}f, ST^{-1} T^{iy}g\big)_{\cH}\big|
\leq e^{- k y^2 + 2 \pi |y|} \big\|ST^{-1}\big\|_{\cB(\cH)} \, \|f\|_{\cH} \, \|g\|_{\cH}    \no \\
& \leq e^{k^{-1} \pi^2} \big\|ST^{-1}\big\|_{\cB(\cH)} \, \|f\|_{\cH} \, \|g\|_{\cH},
\quad y\in\bbR,  \lb{2.51} \\
|\phi(1+ iy)| &= e^{-k y^2} \big|\big(T^{iy}f, \ol{T^{-1}S} T^{iy}g\big)_{\cH}\big|
\leq e^{- k y^2 + 2 \pi |y|} \big\|\ol{T^{-1}S}\big\|_{\cB(\cH)} \,
\|f\|_{\cH} \, \|g\|_{\cH}   \no \\
& \leq e^{k^{-1} \pi^2} \big\|S^* T^{-1}\big\|_{\cB(\cH)} \, \|f\|_{\cH} \, \|g\|_{\cH},
\quad y\in\bbR,  \lb{2.52} \\
|\phi (z)| & = e^{-k |\Im(z)|^2 + k \Re(z) [\Re(z) - 1]} \big|\big(Tf,T^{-1-z} S T^{-2+z} Tg\big)_{\cH}\big|    \no \\
& \leq e^{-k |\Im(z)|^2 + k \Re(z) [\Re(z) - 1] + 2 \pi |\Im(z)|} \big\|T^{-1 - \Re(z)}\big\|_{\cB(\cH)}    \no \\
& \quad \times
\big\|ST^{-1}\big\|_{\cB(\cH)} \, \big\|T^{-1+ \Re(z)}\big\|_{\cB(\cH)} \, \|T f\|_{\cH} \, \|T g\|_{\cH}   \no \\
& \leq C, \quad f, g \in \dom(T), \; z \in \ol\Sigma,    \lb{2.53}
\end{align}
where $C=C(f,g,S,T)>0$ is a finite constant, independent of $z \in \ol\Sigma$.

Applying the Hadamard three-lines estimate \eqref{2.21} to $\phi(\cdot)$ then yields
the first estimate in \eqref{2.44} since $\dom(T)$ is dense in $\cH$ and
$\ol{T^{-z}ST^{-1+z}} \in\cB(\cH)$, $z\in\ol\Sigma$, by \eqref{2.42}.

If in addition $T \geq 0$, one chooses $k=0$ in \eqref{2.50} and uses \eqref{2.24} (instead
of \eqref{2.23}) to arrive at the second estimate in \eqref{2.44}.
\end{proof}

We emphasize that the case $T \geq 0$ (actually, $T \geq \varepsilon I_{\cH}$ for some 
$\varepsilon > 0$) in the estimate \eqref{2.45} was also derived by Huang 
\cite[Lemma~2.1.(a)]{Hu88}.

The results provided thus far naturally extend to the situation where $T^{-z} S T^{-1 + z}$ is
replaced by $T_2^{-z} S T_1^{-1 + z}$ for two self-adjoint operators $T_j$ in $\cH_j$, $j=1,2$.
As an example, we now illustrate this in the context of Theorem \ref{t2.7}.

\begin{hypothesis} \lb{h2.8}
Assume that $T_j$ are self-adjoint operators in $\cH_j$ with
$T_j^{-1} \in \cB(\cH_j)$, $j=1,2$. In addition, suppose that $S$ is a closed operator
mapping $\dom(S) \subseteq \cH_1$ into $\cH_2$ satisfying
\begin{equation}
\dom(S) \supseteq \dom(T_1) \, \text{ and } \,
\dom(S^*) \supseteq \dom(T_2).   \lb{2.54}
\end{equation}
\end{hypothesis}

In particular, Hypothesis \ref{h2.8} implies that
\begin{equation}
ST_1^{-1} \in \cB(\cH_1,\cH_2), \quad S^*T_2^{-1} \in \cB(\cH_2,\cH_1).     \lb{2.55}
\end{equation}

Assuming Hypothesis \ref{h2.8}, one obtains the following corollary of Theorem
\ref{t2.7}:

\begin{corollary} \lb{c2.9}
Assume Hypothesis \ref{h2.8}. Then
$T_2^{-z}ST_1^{-1+z}$  defined on $\dom(T_1)$, $z\in\ol\Sigma$, is closable, and
\begin{align}
\begin{split}
\ol{T_2^{-z}ST_1^{-1+z}} &= T_2^{-i \Im(z)} \big[|S^*|^{\Re(z)} T_2^{- \Re(z)}\big]^* U_S    \\
& \quad \times |S|^{1-\Re(z)} T_1^{-1+\Re(z)} T_1^{i \Im(z)} \in \cB(\cH_1,\cH_2),
\quad z \in \ol \Sigma.   \lb{2.55a}
\end{split}
\end{align}
In addition, given $ k \in (0,\infty)$, one obtains
\begin{align}
& \big\|\ol{T_2^{-z}ST_1^{-1+z}}\big\|_{\cB(\cH_1,\cH_2)}
\leq \big\|ST_1^{-1}\big\|_{\cB(\cH_1,\cH_2)}^{1-\Re(z)} \,
\big\|S^*T_2^{-1}\big\|_{\cB(\cH_2,\cH_1)}^{\Re(z)}    \no \\
& \quad \times \begin{cases} e^{k |\Im(z)|^2 + k \Re(z)[1 - \Re(z)] + k^{-1} \pi^2}, \\
e^{k |\Im(z)|^2 + k \Re(z)[1 - \Re(z)] + (4k)^{-1} \pi^2}, \text{ if $T_1\geq 0$, or $T_2\geq 0$,} \\
1, \text{ if $T_j\geq 0$, $j=1,2$,}
\end{cases} \quad z \in \ol \Sigma,    \lb{2.56}
\end{align}
and
\begin{align}
& \big\|\ol{T_2^{-x}ST_1^{-1+x}}\big\|_{\cB(\cH_1,\cH_2)}
\leq \big\|ST_1^{-1}\big\|_{\cB(\cH_1,\cH_2)}^{1- x} \,
\big\|S^*T_2^{-1}\big\|_{\cB(\cH_2,\cH_1)}^x    \no \\
& \quad \times \begin{cases} e^{2 \pi [x(1-x)]^{1/2}}, \\
e^{\pi [x(1-x)]^{1/2}}, \text{ if $T_1\geq 0$, or $T_2\geq 0$,} \\
1, \text{ if $T_j\geq 0$, $j=1,2$,}
\end{cases} \quad x \in [0,1].    \lb{2.57}
\end{align}
\end{corollary}
\begin{proof}
Consider $\boldsymbol{\cH}:=\cH_1 \oplus \cH_2$ and introduce
\begin{align}
\textbf{S} & =
\begin{pmatrix} 0 & 0\\ S & 0 \\ \end{pmatrix}, \quad
\dom(\textbf{S}) = \dom(S) \oplus \cH_2,       \lb{2.58} \\
\textbf{S}^* & =
\begin{pmatrix} 0 & S^* \\ 0 & 0 \\ \end{pmatrix}, \quad
\dom(\textbf{S}^*) = \cH_1 \oplus \dom(S^*),       \lb{2.58a}
\end{align}
and
\begin{equation}
\textbf{T} = \begin{pmatrix} T_1 & 0\\ 0 & T_2 \\ \end{pmatrix}, \quad
\dom(\textbf{T}) = \dom(T_1) \oplus \dom(T_2).     \lb{2.59}
\end{equation}

Then $\textbf{S}$ is a closed operator in $\boldsymbol{\cH}$ and $\textbf{T}$ is a self-adjoint operator in $\boldsymbol{\cH}$ with bounded inverse given by
\begin{equation}
\textbf{T}^{-1} = \begin{pmatrix} T_1^{-1} & 0\\ 0 & T_2^{-1} \\
\end{pmatrix} \in \cB(\boldsymbol{\cH}).     \lb{2.60}
\end{equation}
Moreover,
\begin{equation}
\dom(\textbf{S})\cap \dom(\textbf{S}^*) = \dom(S)\oplus \dom(S^*)\supseteq
\dom(T_1)\oplus \dom(T_2) = \dom(\textbf{T}),     \lb{2.61}
\end{equation}
that is, the pair $(\textbf{S},\textbf{T})$ satisfies Hypothesis 2.1.

Thus,
\begin{equation}
\textbf{S}\textbf{T}^{-1} = \begin{pmatrix} 0 & 0 \\ ST_1^{-1} & 0\end{pmatrix}
\in \cB(\boldsymbol{\cH}),
\quad
\textbf{S}^*\textbf{T}^{-1} = \begin{pmatrix} 0 & S^* T_2^{-1}
\\ 0 & 0\end{pmatrix} \in \cB(\boldsymbol{\cH}),
\end{equation}
and
\begin{equation}
\big\|\textbf{S}\textbf{T}^{-1}\big\|_{\cB(\boldsymbol{\cH})} = \big\|ST_1^{-1}\big\|_{\cB(\cH_1,\cH_2)},
\quad
\big\|\textbf{S}^*\textbf{T}^{-1}\big\|_{\cB(\boldsymbol{\cH})} = \big\|S^*T_2^{-1}\big\|_{\cB(\cH_2,\cH_1)}.
\lb{2.62}
\end{equation}
Applying Theorem \ref{t2.7} to the pair $(\textbf{S},\textbf{T})$, one obtains \eqref{2.56} and
\eqref{2.57}.
\end{proof}

In the special case, where $\cH_1 = \cH_2 = \cH$, $S$ and $T_1, T_2\ge 0$  are bounded linear
operators in $\cH$, the third estimate \eqref{2.57} recovers Lemma\ 25 in \cite{CPS09}.
On the other hand, the case $x \in [0,1]$ in the first estimate in \eqref{2.57} is a special case of
\eqref{4.33}, which in turn is recorded in \cite[Lemma~16.3]{Ya10}.

For connections with the generalized Heinz inequality and bounds on operators of the type 
$T_2^{-x}ST_1^{-x}$, $x \in [0,1]$, under various conditions on $S$ and $T_j$, $j=1,2$,   
we also refer to \cite[Lemma~2.1]{Hu88}, \cite[Theorem\ 3]{Ka61}, \cite[Theorem\ 6]{Ka62}, 
and the references cited therein. In particular,  \cite[Lemma~2.1]{Hu88} 
predates and extends \cite[Proposition~A.1.(3)]{Le05}. For additional variants on the Heinz 
inequality we refer, for instance, to \cite[Lemma~11, Remark~12]{PS10}, and especially, to 
\cite[Sect.~3.12]{Fu02} and the detailed bibliography collected therein. For interesting 
extensions of the Heinz inequality employing operator monotone functions we also refer to 
\cite{TUU02}, \cite{Uc99}, \cite{Uc05}. (An exhaustive list of references on (extensions of) 
the Heinz inequality is beyond the scope of this short paper due to the enormous amount 
of literature on this subject.)

\section{Interpolation and Trace Ideals Revisited}  \lb{s3}

In this section we recall a powerful result on interpolation theory in connection
with linear operators in the trace ideal spaces $\cB_p(\cH)$, $p\in [1,\infty)$,
originally due to I.\ C.\ Gohberg, M.\ G. Krein, and S.\ G.\ Krein
(cf.\ \cite[p.\ 139]{GK69}) and then apply it to the types of operators studied in Section \ref{s2}. 
For background on trace ideals we refer to \cite[Ch.~I--III]{GK69}, \cite{Sc60}, 
\cite[Chs.~1--3]{Si05}

In particular, we will dwell a bit on a particular case omitted in the discussion
of Gohberg and Krein \cite[Theorem\ III.13.1]{GK69} (see also
\cite[Theorem\ III.5.1]{GK70}):

\begin{theorem} [\cite{GK69}, Theorem\ III.13.1 (see also
\cite{GK70}, Theorem\ III.5.1)]  \lb{t3.1} ${}$ \\
Let $p_0, p_1 \in [1,\infty)$, $p_0 \leq p_1$, and suppose that $A(z)\in\cB(\cH)$, $z\in\ol \Sigma$,
and that $A(\cdot)$ is analytic on $\Sigma$. Assume that for some $C_0, C_1 \in (0,\infty)$,
\begin{equation}
\|A(iy)\|_{\cB_{p_0}(\cH)} \leq C_0, \quad \|A(1+iy)\|_{\cB_{p_1}(\cH)} \leq C_1,
\quad y\in\bbR,     \lb{3.1}
\end{equation}
and suppose that for all $f,g \in \cH$, there exist $C_{f,g}\in\bbR$ and
$a_{f,g}\in [0,\pi)$, such that
\begin{equation}
\sup_{z\in\Sigma}\Big[e^{-a_{f,g}|\Im(z)|}\ln (|(f, A(z)g)_{\cH}|)\Big]
\leq C_{f,g}.     \lb{3.2}
\end{equation}
Then
\begin{equation}
A(z) \in \cB_{p_z}(\cH), \quad \f{1}{p_z} = \f{1 - \Re(z)}{p_0}
+ \f{\Re(z)}{p_1}, \quad z \in \ol \Sigma,     \lb{3.3}
\end{equation}
and
\begin{equation}
\|A(z)\|_{\cB_{p_z}(\cH)}  \leq C_0^{1-\Re(z)} C_1^{\Re(z)},
\quad z\in \ol \Sigma.     \lb{3.4}
\end{equation}
The estimate \eqref{3.4} remains valid for $p_0 \in [1,\infty)$ and $p_1 = \infty$ in 
\eqref{3.1} and \eqref{3.3}.  \\  
In particular, if $p_0=p_1 \in [1,\infty)$, then
\begin{equation}
\|A(z)\|_{\cB_{p_0}(\cH)}  \leq C_0^{1-\Re(z)} C_1^{\Re(z)},
\quad z\in \ol \Sigma.      \lb{3.5}
\end{equation}
\end{theorem}
\begin{proof}
Theorem \ref{t3.1} and its proof is presented by Gohberg and Krein in
\cite[Theorem\ III.13.1]{GK69} under the additional assumption that $p_0 < p_1$. (Moreover, this additional restriction is repeated in \cite[Theorem\ III.5.1]{GK70}, where the result is stated without proof.) However, their proof extends to special case where $p_0=p_1 \in [1,\infty)$ without any difficulties, in fact, it even simplifies a bit. For the convenience of the reader we now present the proof in this particular situation:

Let $F\in\cB(\cH)$ be a finite-rank operator and suppose that
\begin{align}
\begin{split}
& \|F\|_{\cB_{p_0'}(\cH)} = 1, \quad p_0^{-1} + (p_0')^{-1} =1, \, \text{ if }  \, p_0 \in (1,\infty),  \\
& \|F\|_{\cB(\cH)} =1, \, \text{ if } \, p_0 = 1,      \lb{3.6}
\end{split}
\end{align}
and consider the function
\begin{equation}
\varphi (z) = \tr(A(z) F), \quad z \in \ol \Sigma.     \lb{3.7}
\end{equation}
By the assumptions on $A(\cdot)$, $\varphi(\cdot)$ is analytic on $\Sigma$ and
\begin{equation}
\sup_{z\in\Sigma} \Big[e^{-a |\Im(z)|} \ln(|\varphi(z)|)\Big] \leq C    \lb{3.7a}
\end{equation}
for some $a=a(F)\in [0,\pi)$ and $C=C(F)\in\bbR$. In addition, one estimates
\begin{align}
|\varphi (iy)| &\leq \|A(iy) F\|_{\cB_1(\cH)} \leq \|A(iy)\|_{\cB_{p_0}(\cH)} \,
\|F\|_{\cB_{p_0'}(\cH)} \leq C_0, \quad y\in\bbR,      \lb{3.8} \\
|\varphi (1+iy)| &\leq \|A(1+iy) F\|_{\cB_1(\cH)} \leq \|A(1+iy)\|_{\cB_{p_0}(\cH)} \,
\|F\|_{\cB_{p_0'}(\cH)} \leq C_1, \quad y\in\bbR.     \lb{3.9}
\end{align}
Applying Hadamard's three-lines estimate \eqref{2.21} to $\varphi(\cdot)$ then yields
\begin{equation}
|\varphi (z)| = |\tr(A(z)F)| \leq C_0^{1-\Re(z)} C_1^{\Re(z)},
\quad z \in \ol \Sigma. \lb{3.10}
\end{equation}
Next, denoting by $\cF(\cH)$ the set of all finite-rank operators in $\cH$, we recall
that $B\in \cB_{p_0}(\cH)$ if and only if the number $\|B\|_{\cB_{p_0}(\cH)}$ is finite, where
\begin{align}
\|B\|_{\cB_{p_0}(\cH)} = \begin{cases}
\sup_{0 \neq F \in \cF(\cH)} |\tr(BF)| / \|F\|_{\cB_{p_0'}(\cH)}, & p_0 \in (1,\infty), \\
\sup_{0 \neq F \in \cF(\cH)} |\tr(BF)| / \|F\|_{\cB(\cH)}, & p_0 = 1 
\end{cases}     \lb{3.11}
\end{align}
(cf.\ \cite[Lemma\ III.12.1]{GK69}). Thus, \eqref{3.10} implies \eqref{3.5} due to the normalization in \eqref{3.6}.
\end{proof}

Alternatively, Theorem \ref{t3.1} follows from combining Theorems~IX.20, IX.22, and Proposition~8
in \cite{RS75}.

Next we prove a trace ideal analog of Theorem \ref{t2.7}, using Theorem \ref{t3.1}:

\begin{theorem}   \lb{t3.2}
In addition to Hypothesis \ref{h2.1}, let $p\in[1,\infty)$, and assume that
\begin{equation}
ST^{-1} \in \cB_p(\cH), \quad S^*T^{-1} \in \cB_p(\cH).     \lb{3.12}
\end{equation}
Then $T^{-z}ST^{-1+z}$, $z\in\ol\Sigma$, defined on $\dom(T)$, is closable in $\cH$, and
\begin{equation}
\ol{T^{-z}ST^{-1+z}} \in \cB_p(\cH),  \quad  z \in \ol \Sigma.   \lb{3.13} \\
\end{equation}
In addition, given $ k \in (0,\infty)$, one obtains
\begin{align}
\big\|\ol{T^{-z}ST^{-1+z}}\big\|_{\cB_p(\cH)}
& \leq \big\|ST^{-1}\big\|_{\cB_p(\cH)}^{1-\Re(z)} \, \big\|S^*T^{-1}\big\|_{\cB_p(\cH)}^{\Re(z)}    \no \\
& \quad \times \begin{cases} e^{k |\Im(z)|^2 + k \Re(z)[1 - \Re(z)] + k^{-1} \pi^2}, \\
1, \text{ if $T\geq 0$,}
\end{cases}  \quad z \in \ol \Sigma,    \lb{3.14}
\end{align}
and
\begin{align}
\big\|\ol{T^{-x}ST^{-1+x}}\big\|_{\cB_p(\cH)}
& \leq \big\|ST^{-1}\big\|_{\cB_p(\cH)}^{1- x} \, \big\|S^*T^{-1}\big\|_{\cB_p(\cH)}^x
\begin{cases} e^{2 \pi [x(1-x)]^{1/2}}, \\
1, \text{ if $T\geq 0$,}
\end{cases} \quad x \in [0,1].    \lb{3.14a}
\end{align}
In particular, assuming $T \geq 0$ and taking $x=1/2$ in \eqref{3.14a} one obtains
\begin{align}
T^{-1/2}ST^{-1/2} &= \big(T^{-1/2}S^*T^{-1/2}\big)^* \in \cB_p(\cH),  \lb{3.15}
\end{align}
and
\begin{equation}
\big\|T^{-1/2}ST^{-1/2}\big\|_{\cB_p(\cH)} = \big\|T^{-1/2}S^*T^{-1/2} \big\|_{\cB_p(\cH)}
\leq \big\|ST^{-1}\big\|_{\cB_p(\cH)}^{1/2} \, \big\|S^*T^{-1}\big\|_{\cB_p(\cH)}^{1/2}.    \lb{3.17}
\end{equation}
\end{theorem}
\begin{proof}
First we note that Theorem \ref{t2.7} applies and hence \eqref{2.42}, \eqref{2.44}
are at our disposal. Next, we introduce
\begin{equation}
A (z) = e^{k z(z-1)} \ol{T^{-z}ST^{-1+z}},  \quad z \in \ol \Sigma, \; k \in [0,\infty),   \lb{3.18}
\end{equation}
and focus again on $k > 0$ first.

Employing \eqref{2.23} one estimates
\begin{align}
& \|A(iy)\|_{\cB_p(\cH)} = e^{- k y^2} \big\|T^{-iy}ST^{-1}T^{iy}\big\|_{\cB_p(\cH)}
\leq e^{- k y^2 + 2 \pi |y|} \big\|ST^{-1}\big\|_{\cB_p(\cH)}     \no \\
& \quad \leq e^{k^{-1} \pi^2} \big\|ST^{-1}\big\|_{\cB_p(\cH)}, \quad y\in\bbR,  \lb{3.19} \\
& \|A(1+ iy)\|_{\cB_p(\cH)} = e^{- k y^2} \big\|\ol{T^{-1-iy}ST^{iy}}\big\|_{\cB_p(\cH)}
= e^{- k y^2} \big\|T^{-iy}(S^*T^{-1})^*T^{iy}\big\|_{\cB_p(\cH)}   \no \\
& \quad = e^{k^{-1} \pi^2} \big\|S^* T^{-1}\big\|_{\cB_p(\cH)}, \quad y\in\bbR,  \lb{3.20} \\
& \|A (z)\|_{\cB(\cH)} = e^{- k |\Im(z)|^2 + k \Re(z)[\Re(z) - 1]}   \no \\
& \qquad \times \big\|T^{-i \Im(z)}\big(|S^*|^{\Re(z)}T^{-\Re(z)}\big)^* U_S
|S|^{1-\Re(z)} T^{-1+\Re(z)} T^{i \Im(z)}\big\|_{\cB(\cH)}  \no \\
& \quad \leq e^{- k |\Im(z)|^2 + k \Re(z) [\Re(z) - 1] + 2 \pi |\Im(z)|}    \no \\
& \qquad \times \big\||S^*|^{\Re(z)} T^{-\Re(z)}\big\|_{\cB(\cH)} \,
\big\||S|^{1-\Re(z)} T^{-1+\Re(z)}\big\|_{\cB(\cH)}   \no \\
& \quad \leq C, \quad z \in \ol\Sigma,    \lb{3.21}
\end{align}
where $C=C(S,T)>0$ is a finite constant, independent of $z \in \ol\Sigma$, employing
\eqref{2.48} and \eqref{2.48b}. Here again we used the generalized polar
decomposition \eqref{2.41} for $S$ (with $\alpha = \Re(z)$).

Applying the Hadamard three-lines estimate \eqref{3.5} to $A(\cdot)$ then yields the first relation in \eqref{3.13} and the estimate \eqref{3.14}.
\end{proof}

In the special case where $T \geq 0$ and $S=S^* \in \cB(\cH)$, the second estimate \eqref{3.14a}
recovers the result \cite[Lemma\ 15]{Su13} (see also \cite[Lemma~5.10]{LPS10}). For applications 
of \eqref{3.14a} to scattering theory we refer, for instance, to \cite[Appendix~1]{RS77}.

\begin{corollary}  \lb{c3.3}
In addition to Hypothesis \ref{h2.8}, let $p\in [1,\infty)$ and assume that
\begin{equation}
ST_1^{-1} \in \cB_p(\cH_1,\cH_2), \quad S^*T_2^{-1} \in \cB_p(\cH_2,\cH_1).     \lb{3.24}
\end{equation}
Then $T_2^{-z}ST_1^{-1+z}$ defined on $\dom(T_1)$, $z\in\ol\Sigma$, is closable, and
\begin{equation}
\ol{T_2^{-z}ST_1^{-1+z}} \in \cB_p(\cH_1,\cH_2),  \quad  z \in \ol \Sigma.   \lb{3.25} \\
\end{equation}
In addition, given $ k \in (0,\infty)$, one obtains
\begin{align}
& \big\|\ol{T_2^{-z}ST_1^{-1+z}}\big\|_{\cB_p(\cH_1,\cH_2)}
\leq \big\|ST_1^{-1}\big\|_{\cB_p(\cH_1,\cH_2)}^{1-\Re(z)} \,
\big\|S^*T_2^{-1}\big\|_{\cB_p(\cH_2,\cH_1)}^{\Re(z)}    \no \\
& \quad \times \begin{cases} e^{k |\Im(z)|^2 + k \Re(z)[1 - \Re(z)] + k^{-1} \pi^2}, \\
e^{k |\Im(z)|^2 + k \Re(z)[1 - \Re(z)] + (4k)^{-1} \pi^2}, \text{ if $T_1\geq 0$, or $T_2\geq 0$,} \\
1, \text{ if $T_j\geq 0$, $j=1,2$,}
\end{cases} \quad z \in \ol \Sigma,    \lb{3.26}
\end{align}
and
\begin{align}
& \big\|\ol{T_2^{-x}ST_1^{-1+x}}\big\|_{\cB_p(\cH_1,\cH_2)}
\leq \big\|ST_1^{-1}\big\|_{\cB_p(\cH_1,\cH_2)}^{1- x} \,
\big\|S^*T_2^{-1}\big\|_{\cB_p(\cH_2,\cH_1)}^x    \no \\
& \quad \times \begin{cases} e^{2 \pi [x(1-x)]^{1/2}}, \\
e^{\pi [x(1-x)]^{1/2}}, \text{ if $T_1\geq 0$, or $T_2\geq 0$,} \\
1, \text{ if $T_j\geq 0$, $j=1,2$,}
\end{cases} \quad x \in [0,1].    \lb{3.27}
\end{align}
\end{corollary}
\begin{proof}
One can follow the  proof of Corollary \ref{c2.9} step by step replacing
$\cB(\boldsymbol{\cH})$ and $\cB(\cH_j)$ by $\cB_p(\boldsymbol{\cH})$ and
$\cB_p(\cH_j)$, $j=1,2$, respectively, applying Theorem \ref{t3.2} instead of Theorem \ref{t2.7}.
\end{proof}

Finally, we recall the following known result in connection with the ideal $\cB_{\infty}(\cH)$:

\begin{theorem} [\cite{RS78}, p.\ 115--116] \lb{t3.4}
Suppose that $A(z)\in\cB(\cH)$, $z\in\ol\Sigma$, that $A(\cdot)$ is analytic
on $\Sigma$, weakly continuous on $\ol\Sigma$. Assume that for some
$C_0, C_1 \in (0,\infty)$,
\begin{equation}
\|A(iy)\|_{\cB(\cH)} \leq C_0, \quad \|A(1+iy)\|_{\cB(\cH)} \leq C_1,
\quad y\in\bbR,     \lb{3.28}
\end{equation}
and suppose that for all $f,g\in\cH$, there exist $C_{f,g}\in\bbR$ and
$a_{f,g} \in [0,\pi)$, such that
\begin{equation}
\sup_{z\in\Sigma}\Big[e^{-a_{f,g}|\Im(z)|} \ln(|(f,A(z)g)_{\cH}|)\Big] \leq C_{f,g}.
\lb{3.29}
\end{equation}
In addition, suppose that
\begin{equation}
\text{either } \, A(iy)\in\cB_{\infty}(\cH),  \, \text{ or } \, A(1+iy)\in\cB_{\infty}(\cH)),
\quad y\in\bbR.     \lb{3.30}
\end{equation}
Then,
\begin{equation}
A(z) \in \cB_{\infty}(\cH), \quad z \in \Sigma.    \lb{3.31}
\end{equation}
\end{theorem}

A condition of the type \eqref{3.29} has inadvertently been omitted in
\cite[p.~115--116]{RS78}.

\section{Extensions to Sectorial Operators}  \lb{s4}

In this section we revisit Theorems \ref{t2.3}, \ref{t2.7},
\ref{t3.2}, and Corollaries \ref{c2.9}, \ref{c3.3}, and replace
the self-adjointness hypothesis on $T$ by appropriate sectoriality
assumptions.

We start by recalling the definition of a sectorial operator and
refer, for instance, to \cite[Chs.\ 2, 3, 7]{Ha06} and 
\cite[Chs.\ 2, 16]{Ya10} for a detailed treatment.

\begin{definition} \lb{d7.3}
Let $T$ be a densely defined, closed, linear operator in $\cH$ and
denote by $S_{\omega} \subset \bbC$, $\omega \in [0,\pi)$, the
open sector
\begin{equation}  \lb{4.1}
S_{\omega} = \begin{cases} \{z\in\bbC \,|\, z \neq 0, \,
|\arg(z)|<\omega\},
& \omega \in (0,\pi), \\
(0,\infty), & \omega = 0,
\end{cases}
\end{equation}
with vertex at $z=0$ along the positive real axis and opening
angle $2 \omega$. The operator $T$ is called {\it sectorial of
angle $\omega \in [0,\pi)$}, denoted by $T \in \Sect(\omega)$, if
\begin{align}
\begin{split}
& (\alpha) \;\; \sigma(T) \subseteq \ol{S_{\omega}},     \lb{4.2} \\
& (\beta) \text{ for all $\omega' \in (\omega,\pi)$, }  \,
 \sup_{z\in \bbC\backslash \ol{S_{\omega'}}}
\big\| z (T - z I_{\cH})^{-1}\big\|_{\cB(\cH)} < \infty.
\end{split}
\end{align}
One calls
\begin{equation}
\omega_T = \min \{\omega \in [0,\pi] \,| \, T \in \Sect(\omega)\},
\lb{4.3}
\end{equation}
the {\it angle of sectoriality of $T$.}
\end{definition}

For the remainder of this section we assume that $T$ is sectorial (that is, 
$T \in \Sect(\omega)$ for some $\omega \in [0,\pi)$) and that $T^{-1} \in \cB(\cH)$.

Then fractional powers $T^{-z}$, with $\Re(z) > 0$, of $T$ can be
defined by a standard Dunford integral in $\cB(\cH)$ (cf., e.g.,
\cite[Sect.\ 2.7.1]{Ya10})
\begin{equation}
T^{-z} = (2\pi i)^{-1}  \ointclockwise _{\Gamma} d \zeta \, \zeta^{-z} (T-
\zeta I_{\cH})^{-1}, \quad \Re(z) > 0,     \lb{4.7}
\end{equation}
using the principal branch of $\zeta^{-z}$, $\{\zeta \in \bbC \,|\, |\arg(\zeta)| < \pi\}$, by 
excluding the negative real axis, with $\Gamma$ surrounding $\sigma(T)$ clockwise 
in $(\bbC \backslash (-\infty,0]) \cap \rho(T)$ (cf.\ \cite[p.~92]{Ya10} for precise details). 
An important property of  $T^{-z}$ is that
\begin{equation}
\text{$T^{-z}$ is a $\cB(\cH)$-valued analytic semigroup on $\{z
\in \bbC \,|\, \Re(z) > 0\}$.} \lb{4.8}
\end{equation}

Defining imaginary powers of $T$ requires a bit more
care. Following \cite[p.~105]{Ya10}, we introduce the imaginary
powers $T^{is}, s \in \mathbb R,$ of $T$ as follows:
\begin{align}
\begin{split} 
& T^{is} f = \slim_{z \to is, \, \Re(z) > 0} T^{-z} f,    \\
& f \in \dom (T^{is})=\Big\{g \in \cH \, \Big| \,  
\slim_{ z \to is, \, \Re(z) > 0} T^{-z} g\,\, \text{exists} \Big\}. 
\end{split} 
\end{align}
We note that one can define imaginary powers of $T$ also more
explicitly as follows: for $s \in \bbR$, one sets as in \cite[p.~153]{Am95},
\begin{equation}
T^{is}f:=\frac{\sin (\pi i s)}{\pi is
}\int_{0}^{\infty}t^{is}(T + t I_{\cH})^{-2} T f \, dt,  \quad f \in \dom (T).
\end{equation}
Then the operator $T^{is}$ is closable for every $s \in \bbR$ and one defines 
\begin{equation}
T^{is}:=\ol{T^{is}|_{\dom (T)}}, \quad s \in \mathbb R.
\end{equation}

We also note that there are several definitions of the fractional (and
imaginary) powers in the literature, see, for instance, \cite[Section~3.2 and
Proposition~3.5.5]{Ha06}, \cite{KZPS76}, \cite[Section~4]{Lu09}, \cite{MS01}, 
\cite[Section~1]{Tr95}, \cite[Section~4]{Am95}. In our setting, all of these
definitions coincide (cf.\ \cite{BGT13}), and we provided the most straightforward one.

To be able to argue as in previous sections one needs to deal
with sectorial operators having {\it bounded imaginary powers $($BIP\,$)$}.

\begin{definition}
If $T$ is a sectorial operator on $\cH$ such that $T^{-1} \in
\cB(\cH),$ then $T$ is said to have  {\it bounded imaginary
powers} if $T^{is} \in \cB(\cH)$ for all $s \in \bbR$. This is then 
denoted by $T \in \BIP(\cH)$. 
\end{definition}

We recall that if $T$ admits bounded imaginary powers then
$\big\{T^{is}\big\}_{s \in \bbR}$ is a $C_0$-group on $\cH$ (cf.
\cite[Corollary~3.5.7]{Ha06}). Hence, there exist  $\theta \geq 0$
and $N \geq 1$ such that
\begin{equation}\lb{semigroup}
\|T^{is}\|_{\cB(\cH)} \leq N e^{\theta |s|}, \quad s \in \bbR,
\end{equation}
and we write $T \in \BIP (N,\theta)$ in this case. Clearly, 
\begin{equation} 
\BIP (\cH) = \bigcup_{N\ge 1, \, \theta\ge 0} \BIP(N, \theta).
\end{equation} 
We also define the type $\theta_T$ of the $C_0$-group $\big\{T^{is}\big\}_{s\in\bbR}$ by
\begin{equation}\label{thetat}
\theta_T:=\inf \big\{\theta \geq 0 \, \big| \, \text{there exists $N_{\theta} \geq 1$ such that} \, 
\big\|T^{is}\big\|_{\cB(\cH)} \leq N_{\theta}e^{\theta|s|}, \, 
s \in \bbR\big\}.
\end{equation}

The standard example of operators $T$ satisfying $T \in \BIP(\cH)$ (in addition to the situation described in \eqref{2.24}) are provided by strictly positive 
self-adjoint operators bounded from below (in this case $T \in \BIP(1,0)$) and boundedly 
invertible, m-accretive operators $T$ (in this case $T \in \BIP(1,\pi/2)$). One recalls that $T$ 
is said to be {\it m-accretive} (cf.\ \cite[Sect.~C.7]{Ha06}, \cite{Ka61a}, \cite[Sect.~V.3.10]{Ka80}, \cite[Sect.~4.3]{Lu09}, \cite[Ch.~2]{Ta79}) if and only if 
\begin{equation}
\ol{\dom(T)} = \cH, \quad(- \infty, 0) \subset \rho(T), \quad 
\big\|(T + \lambda I_{\cH})^{-1}\big\|_{\cB(\cH)} \leq \lambda^{-1}, \; \lambda > 0.  
\end{equation}

The following extension of \eqref{4.8} will be vital for the remainder of this section: 

\begin{theorem} [See, e.g., 
\cite{Am95}, Theorem~4.7.1] \label{semigr} ${}$ \\ 
If $T \in \BIP(\cH)$ then $\big\{ T^{-z} \, \big| \, \Re (z)\geq 0 \big\}$ is a strongly
continuous semigroup in the closed right half-plane $\{z \in \bbC \, | \, \Re(z)\geq 0\}.$
\end{theorem}

We note that by  \cite[Proposition~7.0.1]{Ha06} (or \cite[p.~101]{Ya10}),
$T \in \Sect(\omega)$ if and only if  $T^* \in \Sect(\omega).$
Moreover, $(T^{z})^*=(T^*)^{\bar z}$ and thus
\begin{equation}\label{adj}
 T \in \BIP(N,\theta) \quad \text{if and only if}\quad   T^* \in \BIP (N,\theta).
\end{equation}

Theorem \ref{semigr} together with \eqref{adj} permits us to
use the three-lines theorem in the present, more general setting
of sectorial operators.

In the special case where $T$ is self-adjoint and strictly
positive in $\cH,$ that is, $T \geq \varepsilon I_{\cH}$ for some
$\varepsilon > 0$, $T^\alpha$, $\alpha \in \bbC$, defined on one
hand as sectorial operators as above, and on the other by the
spectral theorem, coincide (cf., e.g.,\cite[Sect.\ 4.3.1]{Lu09},
\cite[Sect.\ 1.18.10]{Tr95}). In particular,
\begin{equation}
\dom(T^\alpha) = \bigg\{f\in\cH\,\bigg|\, \|T^{\alpha} f\|_{\cH}^2
= \int_{[\varepsilon,\infty]} \lambda^{2 \Re(\alpha)}
d\|E_T(\lambda) f\|_{\cH}^2 < \infty\bigg\}, \quad \alpha\in\bbC,
\lb{4.16}
\end{equation}
in this case. Here $\{E_T(\lambda)\}_{\lambda \in \bbR}$ denotes
the family of spectral projections of $T$.

\medskip

In the remainder of this section, we will use the following set of
assumptions:

\begin{hypothesis} \lb{h4.2}
Assume that $T$ is a  sectorial operator in $\cH$ such that
$T^{-1} \in \cB(\cH)$. In addition, we assume that $S$ is a closed
operator in $\cH$ satisfying
\begin{equation}
\dom(S) \supseteq \dom(T), \quad \dom(S^*) \supseteq \dom(T^*).
\lb{4.17}
\end{equation}
\end{hypothesis}

We start with the analog of Theorem \ref{t2.3}:

\begin{theorem}   \lb{t4.3}
Assume Hypothesis \ref{h4.2}. Then the following facts hold: \\
$(i)$ The operator $T^{-1}ST$ is well-defined on $\dom(T^2)$, and
hence densely defined in $\cH$,
\begin{equation}
\dom\big(T^{-1}ST\big) \supseteq \dom\big(T^2\big).     \lb{4.18}
\end{equation}
$(ii)$ The relation
\begin{equation}
\big(T^{-1}ST\big)^* = T^* S^* (T^*)^{-1}   \lb{4.19}
\end{equation}
holds, and hence $T^*S^*(T^*)^{-1}$ is closed in $\cH$. \\
$(iii)$ One infers that
\begin{equation}
T^{-1}ST \, \text{ is bounded  if and only if } \, (T^{-1}ST)^* =
T^*S^*(T^*)^{-1} \in \cB(\cH).    \lb{4.20}
\end{equation}
In case \eqref{4.20} holds, then
\begin{equation}
\ol{T^{-1}ST} = (T^*S^*(T^*)^{-1})^*, \quad
\big\|\ol{T^{-1}ST}\big\|_{\cB(\cH)} =
\big\|T^*S^*(T^*)^{-1}\big\|_{\cB(\cH)}.    \lb{4.21}
\end{equation}
\end{theorem}
\begin{proof}
Since $\dom\big(T^2\big)$ is an
operator core for $T$ (cf.\ \cite[Theorem~3.1.1]{Ha06}), one can follow the proof of 
Theorem \ref{t2.3} line by line. To illustrate this claim we just mention,
for instance, the analog of \eqref{2.12} which now turns into
\begin{align}
& \big(\big(T^{-1}ST\big)^*f,g\big)_{\cH}  = \big(f,T^{-1}ST
g\big)_{\cH} = \big((T^*)^{-1}f,STg\big)_{\cH}
= \big(S^*(T^*)^{-1}f, Tg\big)_{\cH},  \no \\
& \hspace*{2.8cm} f\in \dom\big(\big(T^{-1}ST\big)^*\big), \; g\in
\dom\big(T^2\big) \subseteq \dom\big(T^{-1}ST\big), \lb{4.22}
\end{align}
and hence once again extends to all $g \in \dom(T)$ as before in
\eqref{2.13}.
\end{proof}

Next, we turn to the analog of Theorem \ref{t2.7} and recall the
notation used in \eqref{semigroup}:

\begin{theorem}   \lb{t4.4}
Assume Hypothesis \ref{h4.2}. If $T \in \BIP(N,\theta),$ then
$T^{-z}ST^{-1+z}$, $z\in\ol\Sigma$, defined on $\dom(T)$, is
closable in $\cH$, and
\begin{align}
\begin{split}
\ol{T^{-z}ST^{-1+z}} &= T^{-i \Im(z)} \big[|S^*|^{\Re(z)} (T^*)^{- \Re(z)}\big]^* U_S    \\
& \quad \times |S|^{1-\Re(z)} T^{-1+\Re(z)} T^{i \Im(z)} \in
\cB(\cH), \quad z \in \ol \Sigma.   \lb{4.23}
\end{split}
\end{align}
In addition, given $k \in (0, \infty)$, one obtains
\begin{align}
\begin{split}
\big\|\ol{T^{-z}ST^{-1+z}}\big\|_{\cB(\cH)}
& \leq N^2 e^{k (\Im(z))^2 + k \Re(z)[1-\Re(z)] + k^{-1} \theta^2}     \\
& \quad \times \big\|ST^{-1}\big\|_{\cB(\cH)}^{1-\Re(z)} \,
\big\|S^*(T^*)^{-1}\big\|_{\cB(\cH)}^{\Re(z)}, \quad z \in \ol
\Sigma,    \lb{4.24}
\end{split}
\end{align}
and
\begin{equation}
\big\|\ol{T^{-x}ST^{-1+x}}\big\|_{\cB(\cH)} \leq N^2 e^{2 \theta
[x(1-x)]^{1/2}} \big\|ST^{-1}\big\|_{\cB(\cH)}^{1-x} \,
\big\|S^*(T^*)^{-1}\big\|_{\cB(\cH)}^{x}, \quad x \in [0,1].
\lb{4.25}
\end{equation}
\end{theorem}
\begin{proof}
Closely examining the first part of the proof of Theorem
\ref{t2.7} based on Heinz's inequality, one notes that everything
up to \eqref{2.49} goes through without any change, implying the
closability of $T^{-z}ST^{-1+z}$ and the validity of \eqref{4.23}.

Next, one defines
\begin{align}
\begin{split}
\phi_k (z) = e^{k z (z-1)} \big(T^*f,T^{-1-z} S T^{-2+z}
Tg\big)_{\cH}
= e^{k z (z-1)} \big(f, \ol{T^{-z}ST^{-1+z}} g\big)_{\cH},&     \\
f \in \dom(T^*), \; g \in \dom(T), \; z \in \ol \Sigma, \; k \in
(0,\infty).&   \lb{4.26}
\end{split}
\end{align}
Then, employing \eqref{semigroup} and \eqref{adj}, one estimates
\begin{align}
& |\phi_k (iy)| = e^{- k y^2} \big|\big((T^*)^{iy}f, ST^{-1} T^{iy}g\big)_{\cH}\big|    \no \\
& \quad \leq e^{- k y^2} N^2 e^{2 \theta |y|}
\big\|ST^{-1}\big\|_{\cB(\cH)} \,
\|f\|_{\cH} \, \|g\|_{\cH}    \no \\
& \quad \leq N^2 e^{k^{-1} \theta^2}
\big\|ST^{-1}\big\|_{\cB(\cH)} \, \|f\|_{\cH} \, \|g\|_{\cH},
\quad y\in\bbR,  \lb{4.27} \\
& |\phi_k (1+ iy)| = e^{- k y^2} \big|\big((T^*)^{iy}f, \ol{T^{-1}S} T^{iy}g\big)_{\cH}\big|   \no \\
& \quad \leq e^{- k y^2} N^2 e^{2 \theta |y|}
\big\|\ol{T^{-1}S}\big\|_{\cB(\cH)} \,
\|f\|_{\cH} \, \|g\|_{\cH}   \no \\
& \quad = e^{- k y^2} N^2 e^{2 \theta |y|} \big\|S^*
(T^*)^{-1}\big\|_{\cB(\cH)} \, \|f\|_{\cH} \, \|g\|_{\cH},
\no \\
& \quad \leq N^2 e^{k^{-1} \theta^2} \big\|S^*
(T^*)^{-1}\big\|_{\cB(\cH)} \, \|f\|_{\cH} \, \|g\|_{\cH},
\quad y\in\bbR,  \lb{4.28} \\
& |\phi_k (z)| = e^{- k (\Im(z))^2 + k \Re(z) [\Re(z) - 1]} \big|\big(T^*f,T^{-1-z} S T^{-2+z} Tg\big)_{\cH}\big|    \no \\
& \quad \leq e^{- k (\Im(z))^2 + k \Re(z) [\Re(z) - 1]}
\big\|T^{-1 - \Re(z) - i \Im(z)}\big\|_{\cB(\cH)} \,
\big\|ST^{-1}\big\|_{\cB(\cH)}     \no \\
& \qquad \times \big\|T^{-1 + \Re(z) + i \Im(z)}\big\|_{\cB(\cH)} \, \|T^* f\|_{\cH} \, \|T g\|_{\cH}   \no \\
& \quad \leq e^{- k (\Im(z))^2 + k \Re(z) [\Re(z) - 1]} N^2 e^{2
\theta |\Im(z)|}
\big\|T^{-1-\Re(z)}\big\|_{\cB(\cH)}  \no \\
& \qquad \times \big\|ST^{-1}\big\|_{\cB(\cH)} \,
\big\|T^{-1+\Re(z)}\big\|_{\cB(\cH)}
\|T^* f\|_{\cH} \, \|T g\|_{\cH}   \no \\
& \quad \leq C_k, \quad f \in \dom(T^*), \; g \in \dom(T), \; z
\in \ol\Sigma,    \lb{4.29}
\end{align}
where $C_k = C_k(f,g,S,T)>0$ is a finite constant, independent of
$z \in \ol\Sigma$.

Applying the Hadamard three-lines estimate \eqref{2.21} to
$\phi(\cdot)$ then yields \eqref{4.24} since $\dom(T)$ and
$\dom(T^*)$ are dense in $\cH$ and $\ol{T^{-z}ST^{-1+z}}
\in\cB(\cH)$, $z\in\ol\Sigma$, by \eqref{4.23}. If $\Im(z)=0$,
optimizing \eqref{4.24} with respect to $k>0$ implies
\eqref{4.25}.
\end{proof}

\begin{remark} \lb{r4.6a} 
We recall that by McIntosh's theorem (cf.\ \cite[Corollary~4.3.5]{Ha06}),
one has
\begin{equation}
\theta_T=\omega_T,
\end{equation}
where   $\omega_T$ and $\theta_T$ are defined by \eqref{4.3} and
\eqref{thetat}, respectively. Thus, in principle, one can use
$\omega_T$ to get estimates cruder than \eqref{4.24}, \eqref{4.25}, 
but then in {\it a priori} terms associated with $T$. However, we decided not to pursue
this here. The same remark also concerns the statements in the remainder of 
this section.
\end{remark}

In the special case where $T \geq 0$ and $S \in \cB(\cH)$, the
estimate \eqref{4.25} recovers \cite[Lemma\ 15]{Su13}.

Again, these results naturally extend to the situation where
$T^{-z} S T^{-1 + z}$ is replaced by $T_2^{-z} S T_1^{-1 + z}$ for
two sectorial operators $T_j$ in $\cH_j$, $j=1,2$, having bounded
imaginary powers, and once more we now illustrate this in the
context of Theorem \ref{t4.4}.

\begin{hypothesis} \lb{h4.5}
Assume that $T_j$ are  sectorial operators in $\cH_j$ such that
$T_j^{-1} \in \cB(\cH_j)$, $j=1,2$. In addition, suppose that $S$ is a
closed operator mapping $\dom(S) \subseteq \cH_1$ into $\cH_2$, 
satisfying
\begin{equation}
\dom(S) \supseteq \dom(T_1) \, \text{ and } \, \dom(S^*) \supseteq
\dom(T_2^*).   \lb{4.30}
\end{equation}
\end{hypothesis}

Then the analog of Corollary \ref{c2.9} reads as follows:

\begin{corollary} \lb{c4.6}
Assume Hypothesis \ref{h4.5}. If  $T_j \in \BIP(N_j,\theta_j)$,
$j=1,2$, then $T_2^{-z}ST_1^{-1+z}$  defined on $\dom(T_1)$,
$z\in\ol\Sigma$, is closable, and
\begin{align}
\begin{split}
\ol{T_2^{-z}ST_1^{-1+z}} &= T_2^{-i \Im(z)} \big[|S^*|^{\Re(z)} (T_2^*)^{- \Re(z)}\big]^* U_S    \\
& \quad \times |S|^{1-\Re(z)} T_1^{-1+\Re(z)} T_1^{i \Im(z)} \in
\cB(\cH_1,\cH_2), \quad z \in \ol \Sigma.   \lb{4.31}
\end{split}
\end{align}
In addition, given $k \in (0, \infty)$, one obtains
\begin{align}
& \big\|\ol{T_2^{-z}ST_1^{-1+z}}\big\|_{\cB(\cH_1,\cH_2)} \leq N_1
N_2
e^{k (\Im(z))^2 + k \Re(z) [1-\Re(z)] + (4k)^{-1} (\theta_1 + \theta_2)^2}   \no \\
& \quad \times \big\|ST_1^{-1}\big\|_{\cB(\cH_1,\cH_2)}^{1-\Re(z)}
\, \big\|S^*(T_2^*)^{-1}\big\|_{\cB(\cH_2,\cH_1)}^{\Re(z)}, \quad
z \in \ol \Sigma,    \lb{4.32}
\end{align}
and
\begin{align}
\begin{split}
& \big\|\ol{T_2^{-x}ST_1^{-1+x}}\big\|_{\cB(\cH_1,\cH_2)}
\leq N_1  N_2  e^{(\theta_1  + \theta_2) [x(1-x)]^{1/2}}     \\
& \quad \times \big\|ST_1^{-1}\big\|_{\cB(\cH_1,\cH_2)}^{1-x} \,
\big\|S^*(T_2^*)^{-1}\big\|_{\cB(\cH_2,\cH_1)}^{x}, \quad x \in
[0,1].    \lb{4.33}
\end{split}
\end{align}
\end{corollary}
\begin{proof}
Again, the $2 \times 2$ block operator formalism introduced in the
proof of Corollary \ref{c2.9} applies to the case at hand.
\end{proof}

We emphasize that \eqref{4.32} is not new, it can be found in
\cite[Lemma~16.3]{Ya10}. Our proof, however, is slightly
different.

Finally, we turn to the analogs of Theorem \ref{t3.2} and
Corollary \ref{c3.3}.

\begin{theorem}   \lb{t4.7}
Assume Hypothesis \ref{h4.2}. Moreover, let $p\in[1,\infty)$, and
suppose that
\begin{equation}
ST^{-1} \in \cB_p(\cH), \quad S^*(T^*)^{-1} \in \cB_p(\cH).
\lb{4.34}
\end{equation}
If $T \in BIP(N,\theta),$ then $T^{-z}ST^{-1+z}$, $z\in\ol\Sigma$,
defined on $\dom(T)$, is closable in $\cH$, and
\begin{equation}
\ol{T^{-z}ST^{-1+z}} \in \cB_p(\cH),  \quad  z \in \ol \Sigma.   \lb{4.35} \\
\end{equation}
In addition, given $k \in (0, \infty)$, one obtains
\begin{align}
\begin{split}
\big\|\ol{T^{-z}ST^{-1+z}}\big\|_{\cB_p(\cH)}
& \leq N^2 e^{k (\Im(z))^2 + k \Re(z) [1-\Re(z)] + k^{-1} \theta^2}     \\
& \quad \times \big\|ST^{-1}\big\|_{\cB_p(\cH)}^{1-\Re(z)} \,
\big\|S^*(T^*)^{-1}\big\|_{\cB_p(\cH)}^{\Re(z)},  \quad z \in \ol
\Sigma,    \lb{4.36}
\end{split}
\end{align}
and
\begin{align}
\begin{split}
\big\|\ol{T^{-x}ST^{-1+x}}\big\|_{\cB_p(\cH)} & \leq N^2 e^{2 \theta [x(1-x)]^{1/2}}  \\
& \quad \times \big\|ST^{-1}\big\|_{\cB_p(\cH)}^{1-x} \,
\big\|S^*(T^*)^{-1}\big\|_{\cB_p(\cH)}^{x}, \quad x \in [0,1].
\lb{4.37}
\end{split}
\end{align}
\end{theorem}
\begin{proof}
First we note that Theorem \ref{t4.4} applies and hence
\eqref{4.23}--\eqref{4.25} are at our disposal. Next, one
introduces
\begin{equation}
A_k (z) = e^{k z (z-1)} \ol{T^{-z}ST^{-1+z}},  \quad z \in \ol
\Sigma, \; k \in (0, \infty),   \lb{4.39}
\end{equation}
and estimates
\begin{align}
& \|A_k (iy)\|_{\cB_p(\cH)} = e^{- k y^2}
\big\|T^{-iy}ST^{-1}T^{iy}\big\|_{\cB_p(\cH)}
\leq e^{- k y^2} N^2 e^{2 \theta |y|} \big\|ST^{-1}\big\|_{\cB_p(\cH)}   \no \\
& \quad \leq N^2 e^{k^{-1} \theta^2} \big\|ST^{-1}\big\|_{\cB_p(\cH)}, \quad y\in\bbR,  \lb{4.40} \\
& \|A_k (1+ iy)\|_{\cB_p(\cH)} = e^{- k y^2} \big\|\ol{T^{-1-iy}ST^{iy}}\big\|_{\cB_p(\cH)}   \no \\
& \quad = e^{- k y^2}
\big\|T^{-iy}(S^*(T^*)^{-1})^*T^{iy}\big\|_{\cB_p(\cH)}
\leq e^{- k y^2} N^2 e^{2 \theta |y|} \big\|S^* (T^*)^{-1}\big\|_{\cB_p(\cH)}   \no \\
& \quad \leq N^2 e^{k^{-1} \theta^2} \big\|S^* (T^*)^{-1}\big\|_{\cB_p(\cH)}, \quad y\in\bbR,  \lb{4.41} \\
& \|A_k (z)\|_{\cB(\cH)} = e^{- k (\Im(z))^2 + k \Re(z) [\Re(z) - 1]}    \no \\
& \qquad \times \big\|T^{-i
\Im(z)}\big(|S^*|^{\Re(z)}(T^*)^{-\Re(z)}\big)^* U_S
|S|^{1-\Re(z)} T^{-1+\Re(z)} T^{i \Im(z)}\big\|_{\cB(\cH)}  \no \\
& \quad \leq e^{- k (\Im(z))^2 + k \Re(z) [\Re(z) - 1]} N^2 e^{2
\theta |\Im(z)|} \,
\big\||S^*|^{\Re(z)} (T^*)^{-\Re(z)}\big\|_{\cB(\cH)} \,    \no \\
& \qquad \times \big\||S|^{1-\Re(z)}
T^{-1+\Re(z)}\big\|_{\cB(\cH)} \leq C_k, \quad z \in \ol\Sigma,
\lb{4.42}
\end{align}
where $C_k=C_k(S,T)>0$ is a finite constant, independent of $z \in
\ol\Sigma$, applying \eqref{2.48} and \eqref{2.48b}. Here we used
again the generalized polar decomposition \eqref{2.41} for $S$
(with $\alpha = \Re(z)$).

Applying the Hadamard three-lines estimate \eqref{3.5} to
$A(\cdot)$ then yields relation \eqref{4.35} and the estimate
\eqref{4.36}. If $\Im(z) = 0$, optimizing \eqref{4.36} with
respect to $k>0$ implies \eqref{4.37}.
\end{proof}

\begin{corollary}  \lb{c4.8}
In addition to Hypothesis \ref{h4.5}, let $p\in [1,\infty)$ and
assume that
\begin{equation}
ST_1^{-1} \in \cB_p(\cH_1,\cH_2), \quad S^*(T_2^*)^{-1} \in
\cB_p(\cH_2,\cH_1).     \lb{4.43}
\end{equation}
If $T_j \in \BIP(N_j,\theta_j)$, $j=1,2,$ then $T_2^{-z}ST_1^{-1+z}$
defined on $\dom(T_1)$, $z\in\ol\Sigma$, is closable, and
\begin{equation}
\ol{T_2^{-z}ST_1^{-1+z}} \in \cB_p(\cH_1,\cH_2),  \quad  z \in \ol \Sigma.   \lb{4.44} \\
\end{equation}
In addition, given $k \in (0, \infty)$, one obtains
\begin{align}
& \big\|\ol{T_2^{-z}ST_1^{-1+z}}\big\|_{\cB_p(\cH_1,\cH_2)} \leq
N_1 N_2
e^{k (\Im(z))^2 + k \Re(z) [1-\Re(z)] + (4k)^{-1} (\theta_1  + \theta_2)^2}   \no \\
& \quad \times
\big\|ST_1^{-1}\big\|_{\cB_p(\cH_1,\cH_2)}^{1-\Re(z)} \,
\big\|S^*(T_2^*)^{-1}\big\|_{\cB_p(\cH_2,\cH_1)}^{\Re(z)}, \quad z
\in \ol \Sigma,    \lb{4.45}
\end{align}
and
\begin{align}
\begin{split}
& \big\|\ol{T_2^{-x}ST_1^{-1+x}}\big\|_{\cB_p(\cH_1,\cH_2)}
\leq N_1 N_2  e^{(\theta_1 + \theta_2) [x(1-x)]^{1/2}}   \\
& \quad \times \big\|ST_1^{-1}\big\|_{\cB_p(\cH_1,\cH_2)}^{1-x} \,
\big\|S^*(T_2^*)^{-1}\big\|_{\cB_p(\cH_2,\cH_1)}^{x},  \quad x \in
[0,1].    \lb{4.46}
\end{split}
\end{align}
\end{corollary}
\begin{proof}
Applying Theorem \ref{t4.7}, one can follow the  proof of
Corollary \ref{c4.6} (see also Corollary \ref{c2.9}) step by step
replacing $\cB(\boldsymbol{\cH})$ and $\cB(\cH_j)$ by
$\cB_p(\boldsymbol{\cH})$ and $\cB_p(\cH_j)$, $j=1,2$,
respectively.
\end{proof}

\medskip

\noindent {\bf Acknowledgments.} We are indebted to Alexander Gomilko for 
very helpful discussions. We also thank Matthias Lesch for valuable 
correspondence. 

F.G.~is indebted to all organizers of the Herrnhut Symposium, ``Operator Semigroups 
meet Complex Analysis, Harmonic Analysis and Mathematical Physics'' (June 3--7, 2013), 
and particularly, to Wolfgang Arendt, Ralph Chill, and Yuri Tomilov, for fostering an 
extraordinarily stimulating atmosphere during the meeting.


\end{document}